\documentclass[review]{elsarticle}
\pdfoutput=1 
\usepackage{lineno,hyperref}
\usepackage{amsmath, amscd, amsfonts, amssymb}
\usepackage{latexsym}
\usepackage{amsmath}
\usepackage{amsthm}
\modulolinenumbers[5]
\usepackage[top=3cm,right=2.8cm,bottom=3cm,left=2.8cm]{geometry}
\usepackage{multicol}
\usepackage{subfigure}
\usepackage{graphicx}
\usepackage{nccmath}
\newtheorem{theorem}{Theorem}[section]
\newtheorem{definition}{Definition}[section]
\newtheorem{lemma}{Lemma}[section]

\begin{document}
\nolinenumbers

\begin{frontmatter}

\title{Numerical convergence and stability analysis for a nonlinear mathematical model of prostate cancer}
\author[mymainaddress]{Farzaneh Nasresfahani\fnref{myfootnote}}
%

\author[mymainaddress]{M.R. Eslahchi\corref{mycorrespondingauthor}}
\cortext[mycorrespondingauthor]{Corresponding author}
\ead{eslahchi@modares.ac.ir}

\address[mymainaddress]{Department of Applied Mathematics, Faculty of Mathematical Sciences, Tarbiat Modares University, P.O. Box 14115-134}

\begin{abstract}
\noindent  The main target of this paper is to present an efficient method to solve a nonlinear free boundary mathematical model of prostate tumor. This model consists of two parabolics, one elliptic and one ordinary differential equations that are coupled together and describe the growth of a prostate tumor. We start our discussion by using the front fixing method to fix the free domain. Then, after employing a nonclassical finite difference and the collocation methods on this model, their stability and convergence  are proved analytically. Finally, some numerical results are considered to show the efficiency of the mentioned methods.
\end{abstract}
\end{frontmatter}
\section{Introduction}
Cancer is the second leading cause of death in the world and there are many types of tumors that can be diagnosed in the human body. There are loads of of cancers which in one sort of category can be divided into three kinds, cancers related to women, to men, and those which consider no gender. Prostate cancer can be regarded as one of the prevalent types of the second kind which is the second most common cancer in men after lung cancer. Approximately one in six men will be diagnosed with prostate cancer during his lifetime, and about
one in 36 will die of prostate cancer \cite{yout}. Most of the information about this type of cancer in the United States of America (USA) originates from the US National Cancer Institute's Surveillance, Epidemiology, and End Results (SEER) program \cite{altekruse2009seer,brawley2012prostate}.\\
Prostate cancer like many other cancers is caused by an abnormal and uncontrolled growth of cells which can either be malignant or benign. Tumor cells of prostate cancer are hormone-sensitive and they crucially depend on male hormone for growth and survival, which is nominated by androgen \cite{marcelli2000androgen}.  The receptor of androgen binds to testosterone and regulates transcription of androgen-responsive genes and many of them stimulate cell proliferation. Cell proliferation is how quickly a cancer cell copies its DNA and divides into 2 cells. So, increasing the level of androgen implies increasing the risk of prostate cancer \cite{debes2004mechanisms}. In this case, the treatment approaches for cancer are to reduce or eliminate testosterone binding to androgen receptors. One of them is a therapy that stops androgen production which is called androgen deprivation therapy (ADT) \cite{huggins2002studies}. Total androgen blockage (TAB) which further combines anti-androgens with ADT is another one which is also used. Due to this issue, scientists focus on hormone therapy of prostate cancer or androgen deprivation therapy. However, after a while, it was concluded experimentally that both ADT and TAB are not so successful in removing all tumor cells and relapse occurs often. This relapse arises because prostate tumors tend to progress to an androgen-independent stage under the selective pressure of androgen ablation therapy and after the positive response to the treatment (especially androgen deprivation therapy). The so-called androgen-independent (AI) cells are considered to be responsible for this relapse. These cells not only are unresponsive to androgen suppression but also convenient to proliferate even in an androgen-poor environment \cite{bruchovsky2000intermittent}. In Figure \ref{fig1} the mutation of AD cells into AI ones and  the process of prostate tumor under the AI relapse and continuous androgen suppression therapy is illustrated \cite{ideta2008mathematical}. It is observable that suppression of prostate cancer is applicable to some extent. 
\begin{figure}
\centering
\includegraphics[scale=.7]{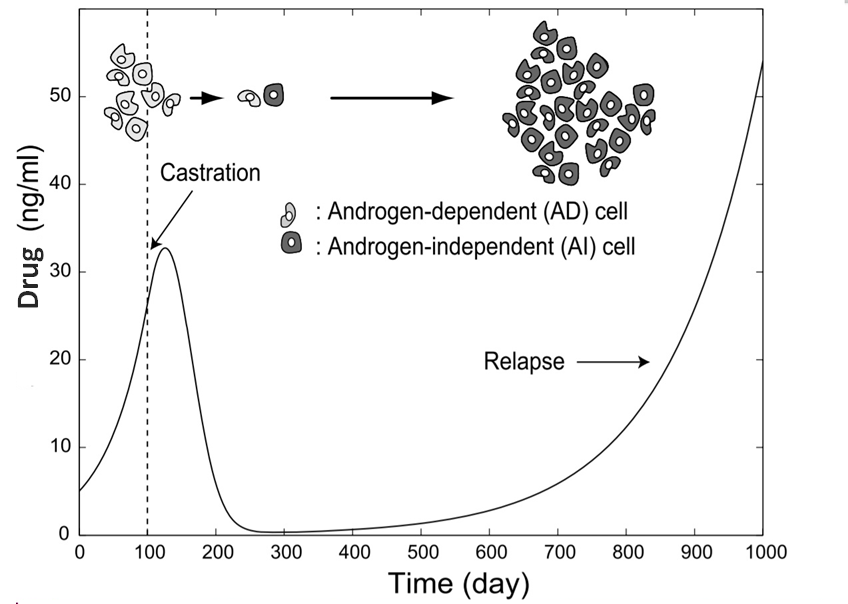}
\caption{\it{Schematic illustration of process of prostate tumor under  androgen suppression therapy \cite{ideta2008mathematical}.}}
\label{fig1}
\end{figure}
Therefore, due to the importance of better recognition of these kinds of prostate cancers (in which relapse occurs) in treatment, their simulation was considered. So, lots of mathematicians were interested in modelling and solving these problems. For instance, the authors of \cite{peng2016prediction} applied a system of nonlinear coupled ordinary differential equations to analyse a mathematical model of the treatment of prostate cancer. In this model, the effects of androgen-deprivation therapy on the prostate tumor cells is described. In \cite{tao2010mathematical} an interesting mathematical model of prostate cancer which investigates the possible mechanism of an AI tumor relapse is developed as well. Also, the author of \cite{jackson2000mathematical,jackson2004math} presented a mathematical model of prostate cancer which is closely related to experimental studies. In terms of solving numerically, there are loads of authors who has worked on numerical solution of mathematical models such as \cite{ramezani2008combined,zhao2017convergence} in which a multi-term time
fractional diffusion equation is solved numerically using a fully-discrete method and convergence and a
hyperbolic equation with an integral condition is solved using finite difference/spectral method respectively. Also, in \cite{esmaili2017application} a free boundary problem which models tumor growth with drug application is numerically investigated using the collocation method and fixed point theory. The authors of \cite{esmailinasr} have presented a fractional diffusion mathematical model of tumor and it has been solved numerically using nonclassical finite difference and collocation method.
An important factor in spectral methods to have stable results is to choose an appropriate set of trial functions. In many cases, this proportion is satisfied using the famous family of classical Jacoby orthogonal polynomials. For instance, in \cite{shamsi2012determination}, the authors apply the Legendre pseudospectral method to determine a control function in parabolic equations. In terms of finite difference method, there are quite a few authors working on different versions of this method. For instance in \cite{dehghan2006finite}, several finite difference schemes are discussed. Also, in \cite{dehghan2017comparison}, the authors investigate a model of tumor using a meshless method.
\subsection{Mathematical model}
One of the models with the mentioned properties is presented in \cite{tao2010mathematical} which is considered in this article and it is as follows
\begin{eqnarray}\nonumber
&\dfrac{\partial P}{\partial t}(r,t)+\dfrac{1}{r^2}\dfrac{\partial}{\partial r}{{(r^2 u(r,t)P(r,t))}}=&\\ \nonumber
&\dfrac{D_p}{r^2}\dfrac{\partial}{\partial r}\left( r^2\dfrac{\partial P}{\partial r}(r,t)\right)+\alpha_p(a(t))P(r,t)-\delta_p(a(t))P(r,t)-(1-I)\beta(a(t))P(r,t),&\\
&P(r,0)=P_0(r),~~\dfrac{\partial P}{\partial r}(0,t)=\dfrac{\partial P}{\partial r}(R(t),t)=0,&\label{initialp}
\end{eqnarray}
\begin{eqnarray}\nonumber
&\dfrac{\partial q}{\partial t}(r,t)+\dfrac{1}{r^2}\dfrac{\partial}{\partial r}{{(r^2 u(r,t)q(r,t))}}=&\\ \nonumber
&\dfrac{D_q}{r^2}\dfrac{\partial}{\partial r}\left( r^2\dfrac{\partial q}{\partial r}(r,t)\right)+\alpha_q(a(t))q(r,t)-\delta_q(a(t))q(r,t)+(1-I)\beta(a(t))p(r,t),&\\
&q(r,0)=q_0(r),~~\dfrac{\partial q}{\partial r}(0,t)=\dfrac{\partial q}{\partial r}(R(t),t)=0,&\label{initialq}
\end{eqnarray}
\begin{eqnarray}\nonumber
&\dfrac{1}{r^2} \dfrac{\partial}{\partial r}\left( r^2 u(r,t)P(r,t)\right)=&\\\nonumber
&\alpha_p (a(t))P(r,t)+1-P(r,t)-\delta_p(a(t))P(r,t)-\delta_q(a(t))(1-P(r,t)),&\\
&u(0,t)=0,&\label{initialu}
\end{eqnarray}
\begin{eqnarray}\nonumber
&\dfrac{\mathrm{d}R}{\mathrm{d}t}=u(R(t),t),&\\
&R(0)=1,\label{initialR}
\end{eqnarray}
where $a(t)$ is the androgen level, $D_p$, $D_q$ are random motility coefficients,  $\alpha_p(a(t))$ and $\alpha_q(a(t))$ are the proliferation rates and $\delta_p(a(t))$ and $\delta_q(a(t))$ are the apoptosis rates of the AD and AI cells respectively. Also,  $\beta (a(t))$ is the mutation rate which indicates the mutation of the AD cells into the AI ones and  $I$ is the intensity of the inhibitors which reduce the mutation rate and its range is from zero to one. i.e. $I=0$ is related to no inhibition of the mutation and $I=1$ is related to the perfect inhibition of the mutation. 
The tumor is given by 
$0<r<R(t) $, where $r$ is measured with the unit of cm, and $t$ is measured with the unit of the day. Also, the variables $P$, $q$ are taken to be functions of $(r,t)$ only in the region $\{(r,t); 0 < r < R(t), t > 0\}$ which are the volume fractions of AD and AI cells.
In this model, the total number of AD and AI cells are considered to be constant per unit volume within the tumor, i.e.,
\begin{equation}\label{constantsum}
P+q=k\equiv \text{constant},
\end{equation}
However, the total combined mass of $P$ and $q$ generally varies in time due to the proliferation and death of the AI and AD cells. This change in total mass gives rise to a velocity field $u$ which depend on the level of $a(t)$.

The most significant issue is to prove the existence and uniqueness of a mathematical model in order to ensure that  there exists only one solution satisfying the given initial and boundary conditions. The existence and uniqueness of this mathematical model are presented in \cite{tao2010mathematical} in detail.

The key factors to choose this mathematical model for solving and analysing are the features of the equations of the model as follows
\begin{itemize}
\item 
The growth of a prostate tumor considering hormone therapy and hypothetical mutation inhibitors is described. To determine more percisely,  This model not only considers the mutation of androgen-dependent (AD) tumor cells into androgen-independent (AI) ones but also introduces inhibition which is assumed to change the mutation rate which makes it more reliable model for prostate cancer.
\item
This model consists of  two coupled parabolic equation which stems from the diffusion treatment of two types of cells (AD and AI) which their proliferation and apoptosis rates are functions of androgen concentration.
\item
As tumor is assumed to be an incompressible fluid, $u$ is the velocity generated by cell proliferation and apoptosis. This feature is illustrated as an elliptic equation that is coupled with the other two parabolic equations.
\item
This model is a free bundary problem and it is rooted from assuming the prostate tumor as a densely packed and radially symmetry sphere of radius $R(t)$. The alteration of the size of the tumor is presented in the model by an ODE which shows the positive correlation between the velocity of radius of the tumor and the value of function $u$ in $R(t)$.
\end{itemize}
In this article, we intend to solve this free boundary nonlinear system of coupled PDEs that model prostate cancer which consists of two parabolics, one elliptic and one ordinary differential equations. For the readers' convenience, we highlight the main goals of this study as follows
\begin{itemize}
\item
As the free boundary problem should be transformed in order to change the model to an appropriate one to use the collocation method and to achieve more comfortable results for numerical analysis, it has been fixed using the front fixing method which is picked among the other means for solving free boundary problems (front fixing, front tracking and fixed domain methods). This choice lies in the fact that the domain of the model is a sphere and the more appropriate method with lower computationaly cost is to eliminate the free bounadry for this problem using the front fixing method (See (\ref{frontfix})).
\item
We have constructed a sequence by applying the finite difference method which converges to the exact solution of coupled partial differential equations (See Theorem \ref{convergencetheorem}).
\item
In each time step, using Taylor theorem,  the problem has changed to linear one and using the collocation method, equations (\ref{transformedp})-(\ref{transformedR}) are solved numerically.
\item
It has been proved that the constructed sequence converges to the exact solution of the problem and also the stability of the method has been proven (See Theorem \ref{stabilitytheorem} and \ref{convergencetheorem}).
\item
Numerical examples are presented to show the efficiency of the presented methods (See Examples 1 and 2).
\end{itemize}
\section{Approximating the solution of the problem}
Due to the diffusion nature of the tumor, the model is a free boundary. This feature can cause some difficulties in applying classical numerical methods and analyzing the convergence and stability which are grafted onto each other. The most notable of which is the need to construct trial functions for spectral methods. In this case, the trial functions depend on time. So, in each time step, the trial functions should be evaluated. These evaluations have computational cost. Meanwhile, due to the mentioned difficulties, it is hard to analyze the convergence and stability in each time step and follow the results. However, to overcome these difficulties one can apply suitable techniques to these mathematical models comprising the front fixing method \cite{wu1997front}. In this technique, the free boundary is fixed by a variable change. Due to the fact that the tumor grows radially symmetric with free boundary, using the front fixing method and a variable changes using a linear transformation $\psi:[R(t),1]\rightarrow=[-1,1]$ we transform the domain to a fixed one by the following variable changes 

\begin{align}\label{frontfix}
\rho:=\dfrac{2r}{R(t)}-1,\quad{v}(\rho ,t)&={u}(\psi^{-1}(\rho),t)={u}(r,t),\quad{P}(\rho ,t)={p}(\psi^{-1}(\rho),t)={p}(r,t).
\end{align}
the free boundary problem  is transformed into a problem with the fixed domain 
\[\{(\rho , t) \mid -1< \rho < 1 , t \geq 0\},\]

Please note that the interval $[-1,1]$ is chosen in order to use classical orthogonal polynomials.

Along with considering the Equations (\ref{initialp}), (\ref{initialq}) and (\ref{initialu}) with assumption (\ref{constantsum}) we have
\begin{eqnarray}\nonumber
&\dfrac{\partial p}{\partial t}(\rho,t)-\dfrac{(\rho+1)v(1,t)}{R(t)}\dfrac{\partial p}{\partial \rho}(\rho,t)-\dfrac{4D_p}{R(t)^2(\rho +1)^2}\dfrac{\partial}{\partial \rho}((\rho +1)^2\dfrac{\partial p(\rho,t)}{\partial \rho})=&\\\nonumber
&1-p(\rho,t)-\delta_p (a(t))(1-p(\rho,t))-(1-I)\beta(a(t))p(\rho ,t),&\\
&p(\rho,0)=p_0(\rho),~~\dfrac{\partial p}{\partial \rho}(-1,t)=\dfrac{\partial p}{\partial \rho}(1,t)=0,&\label{transformedp}
\end{eqnarray}
\begin{eqnarray}\nonumber
&\dfrac{\partial }{\partial \rho}((\rho+1)^2v(\rho ,t)p(\rho,t))=&\\\nonumber
&\dfrac{R(t)(\rho +1)^2}{2}\left( \alpha_p (a(t))p(\rho,t)+1-p(\rho,t)-\delta_p(a(t)) p(\rho,t)-\delta_q(a(t))(1-p(\rho,t))\right),&\\
&v(-1,t)=0,&\label{transformedv}
\end{eqnarray}
\begin{eqnarray}\nonumber
&\dfrac{\mathit{d}R(t)}{\mathit{d}t}=v(1,t),&\\
&R(0)=1.&\label{transformedR}
\end{eqnarray}
The functions $a(t)$, $\alpha_p(a(t))$, $\alpha_q(a(t))$, $\delta_p(a(t))$, $\delta_q(a(t))$ and $\beta(a(t))$ take the following specific forms \cite{jackson2000mathematical,jackson2004math}
\begin{eqnarray*}
&&a(t)=\exp(-bt)+a_s,~~ t\geq 0,\\
&&\alpha_p (a(t))=\theta_1+(1-\theta_1)\dfrac{a(t)}{a(t)+K},\\
&&\delta_p (a(t))=\delta_1\left[ w_1+(1-w_1)\dfrac{a(t)}{a(t)+K}\right],\\
&&\delta_q (a(t))=\delta_2\left[ w_2+(1-w_2)\dfrac{a(t)}{a(t)+K}\right],\\
&&\beta (a(t))=\beta_1\left(1-\dfrac{a(t)}{1+a_s}\right),
\end{eqnarray*}
where $a_s, \, \theta_1 ,\, b,\,K,\, \delta_1,\,\delta_2,\, w_1$ and $w_2$ are positive constants  and the following conditions are assumed
\[
0\leq a_s<1,~~ 0\leq \theta_1<1,~~ \delta_1<\delta_2,~~ w_2<1<w_1,
\]
and the parameter $a_s>0$ corresponds to ADT, and $a_s = 0$ corresponds to TAB.
(For furthure information plaese see  \cite{tao2010mathematical}).\\
Now, we want to approximate the solution of the problem (\ref{transformedp})-(\ref{transformedR}) for $-1 < \rho < 1$ and $0 < t < T$.
Let $t_i := ih\,  (i = 0,\cdots ,M)$ be mesh points, where $h:=\dfrac{T}{M}$ is the time step and $M$ is a positive
integer. Our goal is to solve the problem employing spectral method for the space approximation and the following non-classical discretization of second-order formula for
approximating the time derivative for a given function $z(\rho ,t)$
\begin{equation}\label{time_derivative}
\dfrac{\partial z}{\partial t}(\rho,t_{n+1})=\dfrac{z(\rho,t_{n+1})-z(\rho,t_{n})+\dfrac{z(\rho,t_{n-1})-z(\rho,t_{n})}{3}}{\dfrac{2h}{3}}+\, E_z^{n,1},
\end{equation}
and the following approximation for linearizing the equations
\begin{equation}\label{lenear}
z(\rho, t_{n+1})=2z(\rho,t_n)-z(\rho,t_{n-1})+\, E_z^{n,2},
\end{equation}
where $E_z^{n,1}$ and $E_z^{n,2}$ are the truncation errors and can be easily verified that there is a positive constant $c_1$ by which the following inequality holds
\begin{equation}
\max\{\Vert E_z^{n,1} \Vert_{\infty}, \Vert E_z^{n,2}\Vert_{\infty}\}<c_1 h^2.
\end{equation}
In the following, we have assumed that for a given functions $f(\rho,t)$ and $g(t)$ we have
\[
f_n(\rho)=f(\rho,t_n),
\]
\[
g_n=g(t_n).
\]
Implementing (\ref{time_derivative}) as an approximation of $\dfrac{\partial p(\rho,t_{n+1})}{\partial t}$ and (\ref{lenear}) into (\ref{transformedp}) leads to the following scheme
\begin{eqnarray}\nonumber
&p_{n+1}(\rho)-h^*\left( \dfrac{(\rho+1)(2v_n(1)-v_{n-1}(1)}{R_{n+1}}\right)\dfrac{\partial p_{n+1}(\rho)}{\partial \rho}-&\\\nonumber
&h^*\dfrac{4D_p}{R_{n+1}^2(\rho+1)^2}\dfrac{\partial}{\partial \rho}\left((\rho+1)^2\dfrac{\partial p_{n+1}(\rho)}{\partial\rho}\right)=&\\\nonumber
&h^* (2f(p_n(\rho))-f({p_{n-1}(\rho)}))+p_n(\rho)-\dfrac{p_{n-1}(\rho)-p_n(\rho)}{3}-h^*E_p^n,&\\ \nonumber
&p(0,\rho)=p_0(\rho),~~-1<\rho <1,&\\
&\dfrac{\partial p_{n+1}}{\partial \rho}(-1)=\dfrac{\partial p_{n+1}}{\partial\rho}(1)=0,&
\end{eqnarray}
\begin{eqnarray}\nonumber
&\dfrac{\partial }{\partial \rho}((\rho+1)^2v_n(\rho)p_n(\rho))=&\\ \nonumber
&\dfrac{R_n(\rho +1)^2}{2}\left( \alpha_p (a(t_n))p_n(\rho)+1-p_n(\rho)-\delta_p(a(t_n)) p_n(\rho)-\delta_q(a(t_n))(1-p_n^{ap}(\rho))\right)  ,&\\
&v_n(-1)=0,&
\end{eqnarray}
where $h^*:=\dfrac{2h}{3}$ and
\[f(p_n(\rho)):=1-p_n(\rho)-\delta_p (a(t_n))(1-p_n(\rho))-(1-I)\beta(a(t_n))p_n(\rho).\]
Here $E_p^n$ is obtained by merging the errors of $E_p^{n,1}$ and $E_p^{n,2}$ in which there is a positive constant $c_2$ such that \begin{equation}\label{trunc1}
\Vert E_p^n\Vert_{\infty}<c_2 h^2.
\end{equation}
Now, applying difference formula (\ref{time_derivative}) on (\ref{transformedR}), we have
\begin{equation}\label{khyksl}
R_{n+1}=R_{n}-\dfrac{R_{n-1}-R_n}{3}+h^*v_n(1)-h^* E_R^{n}.
\end{equation}
Due to the fact that there exists a positive constant $c_3$ such that $\Vert E_R^{n}\Vert_{\infty}<c_3h^2$ and considering (\ref{trunc1}) we have
\begin{equation}\label{trunc2}
\max\{\Vert E_p^n\Vert_{\infty},\Vert E_R^{n}\Vert_{\infty}\}<c_4 h^2,
\end{equation}
where $c_4$ is a positive constant.\\
Now we approximate the solution of the problem (\ref{transformedp})-(\ref{transformedR}) by $(p_{n+1}^{ap},R_{n+1}^{ap})$, which is the approximated solution of the following problem
\begin{eqnarray}\nonumber
&p_{n+1}(\rho)-h^*\left( \dfrac{(\rho+1)(2v_n^{ap}(1)-v_{n-1}^{ap}(1)}{R^{ap}_{n+1}}\right)\dfrac{\partial p_{n+1}(\rho)}{\partial \rho}-&\\\nonumber
&h^*\dfrac{4D_p}{(R^{ap}_{n+1})^2(\rho+1)^2}\dfrac{\partial}{\partial \rho}\left((\rho+1)^2\dfrac{\partial p_{n+1}(\rho)}{\partial\rho}\right)&\\\nonumber
&=h^* (2f(p_n^{ap}(\rho))-f({p_{n-1}^{ap}(\rho)}))+p_n^{ap}(\rho)-\dfrac{p_{n-1}^{ap}(\rho)-p_n^{ap}(\rho)}{3},&\\ \nonumber
&p(\rho,0)=p_0(\rho),~~-1<\rho <1,&\\\label{gelebelo}
&\dfrac{\partial p_{n+1}}{\partial \rho}(-1)=\dfrac{\partial p_{n+1}}{\partial \rho}(1)=0,&
\end{eqnarray}
\begin{eqnarray}\nonumber
&\dfrac{\partial }{\partial \rho}((\rho+1)^2v_n^{ap}(\rho)p_n^{ap}(\rho))=&\\\nonumber
&\dfrac{R_n^{ap}(\rho +1)^2}{2}\left( \alpha_p (a(t_n))p_n^{ap}(\rho)+1-p_n^{ap}(\rho)-\delta_p(a(t_n)) p_n^{ap}(\rho)-\delta_q(a(t_n))(1-p_n(\rho))\right),&\\
&v_n(0)=0,&
\end{eqnarray}
\begin{eqnarray}\nonumber
&R_{n+1}^{ap}=R_{n}^{ap}-\dfrac{R_{n-1}^{ap}-R_n^{ap}}{3}+h^*v_n^{ap}(1),&\\
&R_0=1,&\label{wiera}
\end{eqnarray}
where $p_{n+1}^{ap}$ is obtained as an approximated solution of $p_{n+1}$ by solving the equation (\ref{gelebelo})-(\ref{wiera}) employing the collocation method. To implement this method, it is necessary to introduce a set of trial functions. Also, the feature of the mathematical model dictate us to find the solution of the problem in $H^2([-1,1])$. So, the trial functions are as follows
\begin{equation}\label{basisbasis}
\text{span}\{b_0(\rho),b_1(\rho),\cdots ,b_k(\rho)\}=\{u\in \mathbb{P}_{k+2}(\mathbb{R})\Big|~ \dfrac{\partial u}{\partial \rho}\mid_{\rho=-1}=\dfrac{\partial u}{\partial \rho}\mid_{\rho =1}=0\},
\end{equation}
where $\mathbb{P}_{k+2}(\mathbb{R})$ is the space of polynomials of degree at most $k + 2$.

Now, we denote the approximation of $p_{n+1}$ by $p_{n+1}^N$ defined as follows 
\begin{equation*}
p_{n+1}^N(\rho)=\sum_{i=0}^N a_i^{n+1,N}b_i(\rho).
\end{equation*}
In this article, we intend to use  Jacobi orthogonal polynomials as test functions for our theory discussions which generally have the weight function  $w^{\alpha,\beta}(\rho):=(1+\rho)^\alpha(1-\rho)^\beta$ \\
Based on the mentioned point and regarding the feature of the PDE equations in the model and since we intend to use the Legendre orthogonal polynomails (as a member of the big family Jacobi orthogonal polynomials with $\alpha=\beta=1$) as the trial functions, the approximation space is $L^2_w(-1,1)$ where
\begin{equation*}
L^2_w(-1,1)=\{f:\mathbb{R}\longrightarrow \mathbb{C}, \,\,\int_{-1}^1 f(\rho)w(\rho)d\rho<\infty\},
\end{equation*}
and $w(\rho)=(1+\rho)^{\alpha}(1-\rho)^{\beta}$, with the following inner product and norm
\begin{equation*}
(f,g)_{w^{\alpha ,\beta}}=\int_{-1}^1 (1+\rho)^{\alpha}(1-\rho)^{\beta}f(\rho)g(\rho)d\rho,~~ \Vert f\Vert^2_{w^{\alpha ,\beta}}=(f,f)_{w^{\alpha ,\beta}}.
\end{equation*}
The following equation is considered as the one in which it's results is an approximated solution of the equation (\ref{transformedp}).
\begin{eqnarray}\nonumber
&\Pi_N^{0,0} p_{n+1}^N(\rho) -h^*\Pi_N^{0,0}\underbrace{\left(\dfrac{(\rho+1)(2v_n^{ap}(1)-v_{n-1}^{ap}(1))}{R_{n+1}^{ap}} \right)}_{g_n(\rho)}\dfrac{\partial p_{n+1}^N(\rho)}{\partial \rho}-&\\
&\Pi_N^{0,0}\dfrac{4h^*D_p}{(R_{n+1}^{ap})^2(\rho+1)^2}\dfrac{\partial}{\partial \rho}\left( (\rho+1)^2\dfrac{\partial p_{n+1}^N(\rho)}{\partial \rho}\right) =
I_N^{0,0} g_n^*(\rho),&\label{pi}
\end{eqnarray}
where $\Pi_N^{0,0}$ is the orthogonal projection and $I_N^{0,0}$ is the Jacobi-Gauss-Lobatto interpolation operator with respect to $\rho$ and also
\begin{equation}\label{gnstar}
g_n^*(\rho)=\Pi_N^{0,0}\left( p_n^{ap}(\rho)-\dfrac{p_{n-1}^{ap}(\rho)-p_{n}^{ap}(\rho)}{3} \right)+h^*(2f(p_n^{ap}(\rho))-f(p_{n-1}^{ap}(\rho))).
\end{equation}
\textbf{Remark:} Let $f(t)$ and $g(t)$ are two continious functions. Considering (\ref{pi}) stems from the fact that calculating $(a_i,b_i)$ from $I_m^{0,0}f=\Pi_m^{0,0}f=g$ in which $I_m^{0,0} f=\sum_{i=0}^m a_i t^i$ and $\Pi_m^{0,0} g=\sum_{i=0}^m b_i t^i$ results in solving a linear system.
\section{Fully discrete convergence}
In order to prove the convergence of the presented method, we need to use the principle of mathematical induction. In so doing, we want to show that there exist positive constants $p^*$ and $R^*$
such that
\begin{equation*}
\vert p_k^{ap}-p_k\vert<p^*, ~~ \vert R_k^{ap}-R_k\vert<R^*,~~ \forall k=0,1,\cdots ,M, 
\end{equation*}
where $p_k$, $R_k$ are the exact solution of (\ref{transformedp})-(\ref{transformedR}) in $t_k$ respectively. First we suppose that
\begin{equation*}
\vert p_k^{ap}-p_k\vert<p^*, ~~ \vert R_k^{ap}-R_k\vert<R^*,~~ \forall k\leq n< M.
\end{equation*}
Now,  we need to present the following lemma to prove the convergence theorem.
\begin{lemma}\label{lemma1}
Let $p$ be the exact solution of (\ref{transformedp}) on the domain $[-1,1]\times [0,T]$, $p_{n+1}^{ap}=p_{n+1}^{N},~\forall n\geq 1$, in which $N$ is the number of collocation points  and $\dfrac{\partial^2 p}{\partial \rho^2}$ be $C^1$-smooth fumction. Then, for each $-1<\rho<1$, there exist positive constants $k_1$, $k_2$ and $k_3$ such that 
\begin{eqnarray}\nonumber
&(\dfrac{1}{2}-k_1h^*)\Vert \dfrac{\partial{(p_{n+1}^N(\rho)-p_{n+1,1}^N(\rho))}}{\partial \rho}\Vert_{w^{0,0}}^2\leq \dfrac{1}{2}\Vert \dfrac{\partial{(p_{n}^{ap}(\rho)-p_{n,1}^N(\rho))}}{\partial \rho}\Vert_{w^{0,0}}^2+&\\\nonumber
&\sum_{i=0}^n\dfrac{1}{3^i}h^* k_2 (\Vert p_{i,1}^N(\rho)-p_i^{ap}(\rho)\Vert_{w^{0,0}}^2+\Vert R_{i}-R_i^{ap}\Vert_{w^{0,0}}^2)+h^*\Vert E_p^{n,*}\Vert_{w^{0,0}}^2+h^*K(N),&
\end{eqnarray}
where $K(N)$ is the error generated by the spectral method and we have
\begin{equation}\label{epstar}
\Vert E_p^{n,*}\Vert_{\infty} \leq k_3 h^2, ~~\lim_{N\rightarrow \infty}K(N)=0,
\end{equation}
and  $p_1^N$ is a polynomial such that 
\begin{equation*}
I_N^{0,0} p_1^N=p_1^N,~~ \dfrac{\partial p_1^N}{\partial \rho}(-1,t)=0,~~\dfrac{\partial p_1^N}{\partial \rho}(1,t)=0,~~0\leq t\leq T,~~p_1^N(\rho,0)=0, ~~-1\leq \rho\leq 1,
\end{equation*}
and
\begin{equation*}
\lim_{N\rightarrow \infty}\left(\Vert I_N^{0,0}(\dfrac{\partial p}{\partial \rho}-\dfrac{\partial p_1^N}{\partial \rho})\Vert_{w^{0,0}}^2+\Vert I_N^{0,0}p-p_1^N\Vert_{w^{0,0}}^2+\Vert I_N^{0,0} \dfrac{\partial (p-p_1^N)}{(\rho+1)\partial \rho}\Vert_{w^{0,0}}^2+\Vert I_N^{0,0}\dfrac{\partial^2 (p-p_1^N)}{\partial\rho^2}\Vert_{w^{0,0}}^2\right)=0,
\end{equation*}
and also, $p_{i,1}^N(\rho)=p_1^N(\rho,t_i)$.
\end{lemma}
\begin{proof}
Since $\dfrac{\partial^2 p}{\partial \rho^2}$ is a $C^1$-smooth function, for each $N\in \Bbb {N}$, there exists a polynomial $p_1^N$ such that
\begin{equation*}
I_N^{0,0} p_1^N=p_1^N,~~ \dfrac{\partial p_1^N}{\partial \rho}(-1,t)=0,~~\dfrac{\partial p_1^N}{\partial \rho}(1,t)=0,~~0\leq t\leq T,~~p_1^N(\rho,0)=0, ~~-1\leq \rho\leq 1,
\end{equation*}
and
\begin{equation*}
\lim_{N\rightarrow \infty}\left(\Vert I_N^{0,0}(\dfrac{\partial p_1^N}{\partial \rho}-\dfrac{\partial p}{\partial \rho})\Vert_{w^{0,0}}^2+\Vert I_N^{0,0}p-p_1^N\Vert_{w^{0,0}}^2+\Vert I_N^{0,0} \dfrac{\partial (p-p_1^N)}{(\rho+1)\partial \rho}\Vert_{w^{0,0}}^2+\Vert I_N^{0,0}\dfrac{\partial^2 (p-p_1^N)}{\partial\rho^2}\Vert_{w^{0,0}}^2\right)=0.
\end{equation*}
Now by taking the inner product of both sides of (\ref{pi}) we have
\begin{eqnarray*}
&( \Pi_N^{0,0}(p_{n+1}^N(\rho)- p_{n+1,1}^N(\rho)),\dfrac{\partial^2 (p_{n+1}^N(\rho)-p_{n+1,1}^N(\rho))}{\partial \rho^2})_{w^{0,0}}+&\\
&\dfrac{4h^*D_p}{(R^{ap}_{n+1})^2}( \dfrac{1}{(\rho+1)^2} \dfrac{\partial}{\partial \rho}((\rho+1)^2\dfrac{\partial(p_{n+1}^N(\rho)-p_{n+1,1}^N(\rho))}{\partial \rho}),\dfrac{\partial^2 (p_{n+1}^N(\rho)-p_{n+1,1}^N(\rho))}{\partial \rho^2}  )_{w^{0,0}}=&\\
&h^*( \Pi_N^{0,0}g_n(\rho)\dfrac{\partial (p_{n+1}^N(\rho)- p_{n+1,1}^N(\rho))}{\partial \rho},\dfrac{\partial^2 (p_{n+1}^N(\rho)-p_{n+1,1}^N(\rho))}{\partial \rho^2})_{w^{0,0}}+&\\
&( I_N (g_n^*(\rho)-Lp_{n+1,1}^N(\rho)),\dfrac{\partial^2 (p_{n+1}^N(\rho)-p_{n+1,1}^N(\rho))}{\partial \rho^2})_{w^{0,0}},&
\end{eqnarray*}
where $p_{n+1,1}^N(\rho):=p_1^N(\rho,t_{n+1})$ and $L$ is defined as follows
\begin{equation}\label{definL}
L\phi :=\phi-h^*g_n(\rho)\dfrac{\partial \phi}{\partial \rho}-\dfrac{4h^*D_p}{(R^{ap}_{n+1})^2}\left( \dfrac{1}{(\rho+1)^2} \dfrac{\partial}{\partial \rho}\left((\rho+1)^2\dfrac{\partial\phi}{\partial \rho}\right)\right).
\end{equation}
Therefore, using Caushy-Schwarz inequality, there exists a positive $c_5$ such that 
\begin{eqnarray}\nonumber
&\Vert \dfrac{\partial (p_{n+1}^N(\rho)-p_{n+1,1}^N(\rho))}{\partial \rho}\Vert^2_{w^{0,0}}+\dfrac{4h^*D_p}{(R_{n+1}^{ap})^2}\Vert \dfrac{\partial^2(p_{n+1}^N(\rho)-p_{n+1,1}^N(\rho))}{\partial\rho^2} \Vert^2_{w^{0,0}}&\\\nonumber
&+\dfrac{4h^*D_p}{(R_{n+1}^{ap})^2}\Vert  \dfrac{\partial (p_{n+1}^N(\rho)-p_{n+1,1}^N(\rho))}{(\rho+1)\partial \rho}\Vert_{w^{0,0}}^2 \leq\vert(I_N(g_n^*(\rho)-Lp_{n+1,1}^N(\rho)),\dfrac{\partial^2(p_{n+1}^N(\rho)-p_{n+1,1}^N(\rho))}{\partial \rho^2})_{w^{0,0}}\vert +&\\
&h^*c_5 \Vert \dfrac{\partial (p_{n+1}^N(\rho)-p_{n+1,1}^N(\rho))}{\partial \rho}\Vert_{w^{0,0}}^2.&\label{inequality1}
\end{eqnarray}
In addition, from (\ref{gnstar}) we obtain that
\begin{eqnarray}\nonumber
&\vert(I_N^{0,0}(g_n^*(\rho)-Lp_{n+1,1}^N(\rho)),\dfrac{\partial^2 (p_{n+1}^N(\rho)-p_{n+1,1}^N(\rho))}{\partial \rho^2})_{w^{0,0}} \vert \leq &\\\nonumber
&\vert {g_n^1} ,\dfrac{\partial^2 (p_{n+1}^N(\rho)-p_{n+1,1}^N(\rho))}{\partial \rho^2})_{w^{0,0}}\vert + &\\
&\vert (p_n^{ap}(\rho)-p_{n,1}^{N}(\rho)-\dfrac{p_{n-1}^{ap}(\rho)-p_{n-1,1}^N(\rho)-p_n^{ap}(\rho)+p_{n,1}^{N}(\rho)}{3},\dfrac{\partial^2 (p_{n+1}^N(\rho)-p_{n+1,1}^N(\rho))}{\partial \rho^2})_{w^{0,0}}\vert ,&\label{inequality2}
\end{eqnarray}
where 
\begin{equation}
g_n^1:=(I_N^{0,0}(Lp_{n+1,1}^N(\rho)-p_{n,1}^N(\rho)+\dfrac{p_{n-1,1}^N(\rho)-p_{n,1}^N(\rho)}{3})-h^*I_N^{0,0}(2f(p_n^{ap}(\rho))-f(p_{n-1}^{ap}(\rho))).
\end{equation}
Therefore, from (\ref{inequality1}) and (\ref{inequality2}) and using Cauchy-Schwarz and Young inequalities we get that there exist positive constants $c_6$ and $c_7$  such that
\begin{eqnarray}\nonumber
&\dfrac{1}{2}\Vert\dfrac{\partial (p_{n+1}^N(\rho)-p_{n+1,1}^N(\rho))}{\partial \rho}\Vert_{w^{0,0}}^2+\dfrac{3h^*D_p}{(R_{n+1}^{ap})^2}\Vert \dfrac{\partial^2 (p_{n+1}^N(\rho)-p_{n+1,1}^N(\rho))}{\partial \rho^2}\Vert_{w^{0,0}}^2+&\\\nonumber
&\dfrac{4h^*D_p}{(R_{n+1}^{ap})^2}\Vert \dfrac{\partial (p_{n+1}^N(\rho)-p_{n+1,1}^N(\rho))}{(\rho+1)\partial \rho}\Vert_{w^{0,0}}^2\leq c_7h^*\Vert  2f(p_n(\rho))-f(p_{n-1}(\rho))+E_p^{n,2}-2f(p_n^{ap}(\rho))+f(p_{n-1}^{ap}(\rho)) \Vert_{w^{0,0}}^2+&\\\nonumber
&h^*\vert (L_n^2 p_{1}^N(\rho)-L_n^2 p(\rho),\dfrac{\partial^2(p_{n+1}^N(\rho)-p_{n+1,1}^N(\rho))}{\partial \rho^2})_{w^{0,0}} \vert+&\\\nonumber
&c_6h^*\Vert \dfrac{\partial (p_{n+1}^N(\rho)-p_{n+1,1}^N(\rho))}{\partial \rho}\Vert_{w^{0,0}}^2+\dfrac{1}{2}\Vert \dfrac{\partial (p_n^{ap}(\rho)-p_{n,1}^{ap}(\rho))}{\partial \rho}\Vert_{w^{0,0}}^2+&\\
&\dfrac{1}{3}\Big| ( \dfrac{\partial(p_{n-1}^{ap}(\rho)-p_{n-1,1}^N(\rho)-p_n^{ap}(\rho)+p_{n,1}^N(\rho))}{\partial \rho},\dfrac{\partial(p_{n+1}^N(\rho)-p_{n+1,1}^N(\rho))}{\partial \rho})_{w^{0,0}}\big|, &\label{instead}
\end{eqnarray}
in which, for a given function $u(\rho,t)$
\begin{equation*}
L_n^2u(\rho,t)=Lu_{n+1}(\rho)-u_n(\rho)+\dfrac{u_{n-1}(\rho)-u_n(\rho)}{3},
\end{equation*}
and $L$ is defined in (\ref{definL}). On the other hand, from (\ref{pi}) it can be easily conclude that there exist positive constants $c_8$ and $c_9$ such that 
\begin{eqnarray}\nonumber
&\vert (\dfrac{\partial (p_n^{ap}(\rho)-p_{n,1}^N(\rho)-p_{n-1}^{ap}(\rho)+p_{n-1,1}^N(\rho))}{\partial \rho},\dfrac{\partial(p_{n+1}^N(\rho)-p_{n+1,1}^N(\rho))}{\partial \rho})_{w^{0,0}}\vert \leq &\\\nonumber
&(\dfrac{6h^*D_p}{(R_{n+1}^{ap})^2})\Vert \dfrac{\partial^2 (p_{n+1}^N(\rho)-p_{n+1,1}^N(\rho))}{\partial \rho^2}\Vert_{w^{0,0}}+ \dfrac{h^*}{4}\vert( L_{i-1}^2 p_{1}^N(\rho)-L_{i-1}^2p(\rho),\dfrac{\partial^2(p_{n+1}^N(\rho)-p_{n+1,1}^N(\rho))}{\partial \rho^2})_{w^{0,0}}\vert +&\\ \nonumber
&\dfrac{h^*}{2}\vert( L_i^2 p_{1}^N(\rho)-L_{i}^2 p(\rho) ,\dfrac{\partial^2(p_{n+1}^N(\rho)-p_{n+1,1}^N(\rho))}{\partial \rho^2})_{w^{0,0}}\vert +&\\
&\dfrac{1}{3}\Big|( \dfrac{\partial (p_{n-1}^{ap}(\rho)-p_{n-1,1}^N(\rho)-p_{n-2}^{ap}(\rho)+p_{n-2,1}^N(\rho))}{\partial \rho},\dfrac{\partial(p_{n+1}^N(\rho)-p_{n+1,1}^N(\rho))}{\partial \rho})_{w^{0,0}} \Big| .& \label{recurence}
\end{eqnarray}
Then, by applying the recurence relation (\ref{recurence}) repeatedly in (\ref{instead}) one can deduce that there exist constants $c_{10}$ and $c_{11}$ such that we have
\begin{eqnarray}\nonumber
&(\dfrac{1}{2}-c_{10}h^*)\Vert \dfrac{\partial{(p_{n+1}^N(\rho)-p_{n+1,1}^N(\rho))}}{\partial \rho}\Vert_{w^{0,0}}^2\leq \dfrac{1}{2}\Vert \dfrac{\partial{(p_{n}^{ap}(\rho)-p_{n,1}^N(\rho))}}{\partial \rho}\Vert_{w^{0,0}}^2+&\\
&\sum_{i=0}^n\dfrac{1}{3^i}h^* c_{11} (\Vert p_{i,1}^N(\rho)-p_i^{ap}(\rho)\Vert_{w^{0,0}}^2+\Vert R_{i}-R_i^{ap}\Vert_{w^{0,0}}^2)+h^*\Vert E_p^{n,*}\Vert_{w^{0,0}}^2+h^*K(N),&\label{prooflemma}
\end{eqnarray}
where 
\begin{equation}\label{khata}
\Vert E_p^{n,*}\Vert_{\infty}<c^*{h}^2,~~\lim_{N\rightarrow \infty} K(N)=0,
\end{equation}
and $c^*$ is a positive constant.
\end{proof}
\begin{theorem}\label{convergencetheorem}
Let $p_{n+1}^{ap}=p_{n+1}^N$. Under the assumption of Lemma \ref{lemma1}, there exist positive constants $M_3$ and $M_4$ such that
\begin{equation*}
\max_{k=0,\cdots ,n+1} \{\xi_k\}\leq M_3 (e^{M_4 T})({h}^2+(K(N)^{\dfrac{1}{2}})),
\end{equation*}
where
\begin{eqnarray*}
&\xi_k=\Vert \dfrac{\partial (p_k^{ap}(\rho)-p_k(\rho))}{\partial \rho}\Vert_{w^{0,0}}+\vert R_k^{ap}-R_k\vert,&
\end{eqnarray*}
and $K(N)$ is the error generated by the spectral method so that
\begin{eqnarray*}
&\lim_{N\rightarrow \infty }K(N)=0.&
\end{eqnarray*}
\end{theorem}
\begin{proof}
Using (\ref{prooflemma}) one can conclude that there exists a positive constant $M_1$ such that
\begin{equation*}
\max_{k=0,\cdots, n+1}\{\phi_k\}\leq (1+M_1h^*)\phi_n+(1+M_1h^*)^2\max_{k=0,1,\cdots ,n}\{\phi_k\}+(1+M_1h^*)(h^*\Vert E_p^{n,*}\Vert_{w^{0,0}}^2+h^*K(N)),
\end{equation*}
where
\begin{eqnarray*}
&\phi_k=\Vert \dfrac{\partial (p_k^{ap}(\rho)-p_k(\rho))}{\partial \rho}\Vert_{w^{0,0}}^2+\vert R_k^{ap}-R_k\vert^2, &\\
&\lim_{N\rightarrow \infty }K(N)=0,&
\end{eqnarray*}
and $\Vert E_p^{n,*}\Vert_{\infty}$ is defined in (\ref{khata}).
 Therefore, there exists a positive constant $M_2$ such that 
\begin{equation*}
\max_{k=0,\cdots ,n+1}\{\phi_k\}\leq (1+M_2h^*)^{n+1}\phi_0 +\Big|\dfrac{(1+M_2h^*)^{n+1}-1}{M_2h^*}\Big|(h^* (h^2)+h^*K(N)).
\end{equation*}
Finally, it may be concluded that there exist positive  constants $M_3$ and $M_4$ such that
\begin{equation}\label{xi}
\max_{k=0,\cdots ,n+1} \{\xi_k\}\leq M_3 (e^{M_4 T})({h}^2+(K(N))^{\dfrac{1}{2}}).
\end{equation}
where
\begin{eqnarray*}
&\xi_k=\Vert \dfrac{\partial (p_k^{ap}(\rho)-p_k(\rho))}{\partial \rho}\Vert_{w^{0,0}}+\vert R_k^{ap}-R_k\vert,&\\
&\lim_{N\rightarrow \infty }K(N)=0.&
\end{eqnarray*}
\end{proof}
Employing the general sobolev inequality, there exists a positive constant $M_5$ such that
\begin{equation}\label{generalsobolev}
\vert p_{n+1}^{ap}(\rho)-p_{n+1}(\rho)\vert \leq M_5\Vert p_{n+1}^{ap}(\rho)-p_{n+1}(\rho)\Vert_{w^{0,0}}.
\end{equation}
Using (\ref{xi}) and (\ref{generalsobolev}), we can select proper $N$ and $h$ such that
\begin{equation}\label{pstar}
\vert p_{n+1}^{ap}(\rho)-p_{n+1}(\rho)\vert \leq p^*.
\end{equation}
So, by considering (\ref{pstar}) and Theorem \ref{convergencetheorem}, we can conclude that the sequence $\{p_n^{ap},R_n^{ap}\}_{n=0}^{\infty}$ converges to the exact solution of problem (\ref{transformedp})-(\ref{transformedR}) on $[-1,1]\times [0,T]$.
\section{Stability}

This section is presented to prove the stability of the presented method. Partial differential equations are well-known to be stable if the small perturbations in the right-hand side of the equation can lead to arbitrarily small changes in the solution \cite{corcoran2018stability,manchanda2020mathematical}. So, to prove the stability, we first need to construct a perturbed problem using the functions $\epsilon_1(\rho,t)$ and $\epsilon_2(\rho,t)$  as follows
\begin{eqnarray}\nonumber
&\dfrac{\partial p}{\partial t}(\rho,t)-\dfrac{(\rho+1)v(1,t)}{R(t)}\dfrac{\partial p}{\partial \rho}(\rho,t)-\dfrac{4D_p}{R(t)^2(\rho +1)^2}\dfrac{\partial}{\partial \rho}((\rho +1)^2\dfrac{\partial p(\rho,t)}{\partial \rho})=&\\\nonumber
&1-p(\rho,t)-\delta_p (a(t))(1-p(\rho,t))-(1-I)\beta(a(t))p(\rho ,t)+\epsilon_1(\rho,t),&\\
&p(\rho,0)=p_0(\rho),~~\dfrac{\partial p}{\partial \rho}(-1,t)=\dfrac{\partial p}{\partial \rho}(1,t)=0,&\label{purturbedp}
\end{eqnarray}
\begin{eqnarray}\nonumber
&\dfrac{\partial }{\partial \rho}((\rho+1)^2v(\rho ,t)p(\rho,t))=&\\\nonumber
&\dfrac{R(t)(\rho +1)^2}{2}\left( \alpha_p (a(t))p(\rho,t)+1-p(\rho,t)-\delta_p(a(t)) p(\rho,t)-\delta_q(a(t))(1-p(\rho,t))\right)+\epsilon_2(\rho,t),&\\
&v(0,t)=0,&
\end{eqnarray}
\begin{eqnarray}\nonumber
&\dfrac{\mathit{d}R(t)}{\mathit{d}t}=v(1,t),&\\
&R(0)=1.&\label{perturbedR}
\end{eqnarray}
Now, considering this purturbed model of (\ref{transformedp})-(\ref{transformedR}), the stability of the presented method is proved in the following theorem.
\begin{theorem}\label{stabilitytheorem}
Let $\epsilon^*$ be a positive constant and $\vert \epsilon_i\vert <\epsilon^*,~(i=1, 2)$. Then under the assumptions of Lemma \ref{lemma1}, there exist positive constants $M_3^*$ and $M_4^*$ such that
\[\max _{k=0,\cdots,n+1}\{\xi_k\}\leq M_3^*(e^{M_4^*T})({h}^2+\epsilon^*+(K(N))^{\dfrac{1}{2}}),\]
where 
$$\xi_k=\Vert \dfrac{\partial (p_k^{ap}(\rho)-p_k(\rho))}{\partial \rho}\Vert_{w^{0,0}}+\vert R_k^{ap}-R_k\vert.$$
and $K(N)$ is the error generated by the spectral method so that
\begin{eqnarray*}
&\lim_{N\rightarrow \infty }K(N)=0.&
\end{eqnarray*}
\end{theorem}
Since the perturbed terms are bounded, we can consider them as a part of the source term and a lemma similar to lemma 3.1 can be proved for the perturbed model.  Now we prove the theorem as follows
\begin{proof}
If we solve the perturbed problem (\ref{purturbedp})-(\ref{perturbedR}) using the presented method, one can conclude that there exists a positive constant $M_1^*$ such that
\begin{equation*}
\max_{k=0,\cdots, n+1}\{\phi_k\}\leq (1+M_1^*h^*)\phi_n+(1+M_1^*h^*)^2\max_{k=0,1,\cdots ,n}\{\phi_k\}+(1+M_1^*h^*)(h^*(h^2)+h^*K(N)+h^*\epsilon^*),
\end{equation*}
where
\begin{eqnarray*}
&\phi_k=\Vert \dfrac{\partial (p_k^{ap}(\rho)-p_k(\rho))}{\partial \rho}\Vert_{w^{0,0}}^2+\vert R_k^{ap}-R_k\vert^2, &\\
&\lim_{N\rightarrow \infty }K(N)=0.&
\end{eqnarray*}
So, there is a positive constant $M_2^*$ such that by applying the above recurrence relation we get
\begin{equation*}
\max_{k=0,\cdots ,n+1}\{\phi_k\}\leq (1+M_2^*h^*)^{n+1}\phi_0 +\Big|\dfrac{(1+M_2^*h^*)^{n+1}-1}{M_2^*h^*}\Big|(h^*\Vert E_p^{n,*}\Vert_{w^{0,0}}^2+h^*K(N)+h^*\epsilon^*).
\end{equation*}
Finally, it is concluded that there exist positive constants $M_3^*$ and $M_4^*$ such that
\begin{equation}
\max_{k=0,\cdots ,n+1} \{\xi_k\}\leq M_3^* (e^{M_4^* T})({h}^2+(K(N))^{\dfrac{1}{2}}+\epsilon^*),
\end{equation}
where
\begin{eqnarray*}
&\xi_k=\Vert \dfrac{\partial (p_k^{ap}(\rho)-p_k(\rho))}{\partial \rho}\Vert_{w^{0,0}}+\vert R_k^{ap}-R_k\vert,&\\
&\lim_{N\rightarrow \infty }K(N)=0.&
\end{eqnarray*}
\end{proof}
\section{Numerical experiment}
\textbf{Example 1.} Consider the following problem
\begin{eqnarray}\nonumber
&\dfrac{\partial p}{\partial t}(\rho,t)-\dfrac{(\rho+1)v(1,t)}{R(t)}\dfrac{\partial p}{\partial \rho}(\rho,t)-\dfrac{4D_p}{R(t)^2(\rho +1)^2}\dfrac{\partial}{\partial \rho}((\rho +1)^2\dfrac{\partial p(\rho,t)}{\partial \rho})=&\\\nonumber
&1-p(\rho,t)-\delta_p (a(t))(1-p(\rho,t))-(1-I)\beta(a(t))p(\rho ,t)+f_p(\rho,t),&\\
&p(\rho,0)=p_0(\rho),~~\dfrac{\partial p}{\partial \rho}(-1,t)=\dfrac{\partial p}{\partial \rho}(1,t)=0,&\label{ex11}
\end{eqnarray}
\begin{eqnarray}\nonumber
&\dfrac{\partial }{\partial \rho}((\rho+1)^2v(\rho ,t)p(\rho,t))=&\\\nonumber
&\dfrac{R(t)(\rho +1)^2}{2}\left( \alpha_p (a(t))p(\rho,t)+1-p(\rho,t)-\delta_p(a(t)) p(\rho,t)-\delta_q(a(t))(1-p(\rho,t))\right)+f_v(\rho,t),&\\
&v(-1,t)=0,&
\end{eqnarray}
\begin{eqnarray}\nonumber
&\dfrac{\mathit{d}R(t)}{\mathit{d}t}=v(1,t),&\\
&R(0)=1,&\label{ex12}
\end{eqnarray}
and the exact solutions of the above model are as follows
\begin{equation*}
p(\rho,t)=(\exp(t) +1)(\frac{x^3}{3}-x),~~R(t)=1/(t+1),~~ v(\rho,t)=-\dfrac{(\exp(x+1)+1)}{((\exp(2)+1)(t+1)^2)}.  
\end{equation*}
In this section, the main goal is to investigate the numerical solution applied on the model of prostate tumor. We solve the model of the tumor by applying the finite difference/collocation method. In so doing, we should construct trial functions which satisfy the boundary conditions.  For this purpose, a linear combination of  Legendre polynomials is used to approximate the function $p(\rho,t)$ in the form of (\ref{basisbasis}) as follows
\begin{equation*}
p_{n+1}^N(\rho)=\sum_{i=0}^N l_i^{n+1,N}b_i(\rho),
\end{equation*}
where 
\begin{equation}
b_i(\rho)=J^{0,0}_i(\rho)-\dfrac{i(i+1)}{(i+2)(i+3)}J^{0,0}_{i+2}(\rho), ~~ i=1,2,\cdots ,N.
\end{equation}
where $J^{\alpha,\beta}_n$ is the Jacobi orthogonal polynomial of degree $n$ with respect to the weight $(1-x)^{\alpha}(1 + x)^{\beta}$ on the interval $[-1, 1] $. Also, the Gauss-quadrature points $\{x_i^{0,0}\}_{i=1}^N$ are considered as collocation points. The typical parameter values for our numerical simulation are: 
$w_1=0.35,~ w_2=0.1,~\delta_1=0.8245,~\delta_2=1.035,~\theta_1=0.2,~K=1,~a_s=0,~b=1,~\beta_1=0.1,~D_p=1,~I=1,~b=1,$ which are taken from \cite{tao2010mathematical}. 
\begin{definition}
A sequence $\{x_n\}_{n=1}^{\infty}$ is said to converge to $x$ with order $s$ if there exists a positive constant $C$ such that $\vert x_n-x\vert \leq Cn^{-s} ,~~\forall n$.
This can be written as $\vert x_n-x\vert=\mathcal{O}(n^{-s})$.  A practical method to calculate the rate of convergence for a discretization method is to use the following formula
\begin{equation}\label{ratio_remark}
s\approx \dfrac{\log_e(e_{n_2}/e_{n_1})}{\log_e(n_1/n_2)},
\end{equation}
where $e_{n_1}$ and $e_{n_2}$ denote the errors with respect to the step sizes $\dfrac{1}{n_1}$ and $\dfrac{1}{n_2}$, respectively \cite{gautschi1997numerical}.
\end{definition}
We assess the accuracy of the finite difference method by reporting the following error
\[E_{\infty}(p)=\max\{\vert \widehat{p}(x_i,t_j)-p_{N,M}(x_i,t_j)\vert,\,\, i=1,\cdots N,\,\, j=1,\cdots ,M \},\]
in which the $\widehat{p}(x_i,t_j)$ in the problems with exact solution is the exact solution at $x=x_i$ and $t=t_j$ and in the problems without the exact solution is the solution of the problem using finite difference with an appropriate time steps and $N$  collocation points which is chosen to be as an exact solution for obtaining the error. Also in general $p_{N,M}(x_i,t_j)$ is the solution of the problem using $M$ time steps and $N$ collocation points.

We carried out the numerical computations applying the \textbf{MATLAB 2018a} program using a computer with the Intel Core i7 processor (2.90 GHz, 4 physical cores).\\
In Figure \ref{TEN100}, we have plotted the graph of the following error function for  $N=100$.
\begin{equation*}
e_n^{p,M,N}=p_n^{ap}(\rho)-p_n(\rho).
\end{equation*}
\begin{figure}
\centering
\includegraphics[scale=.4]{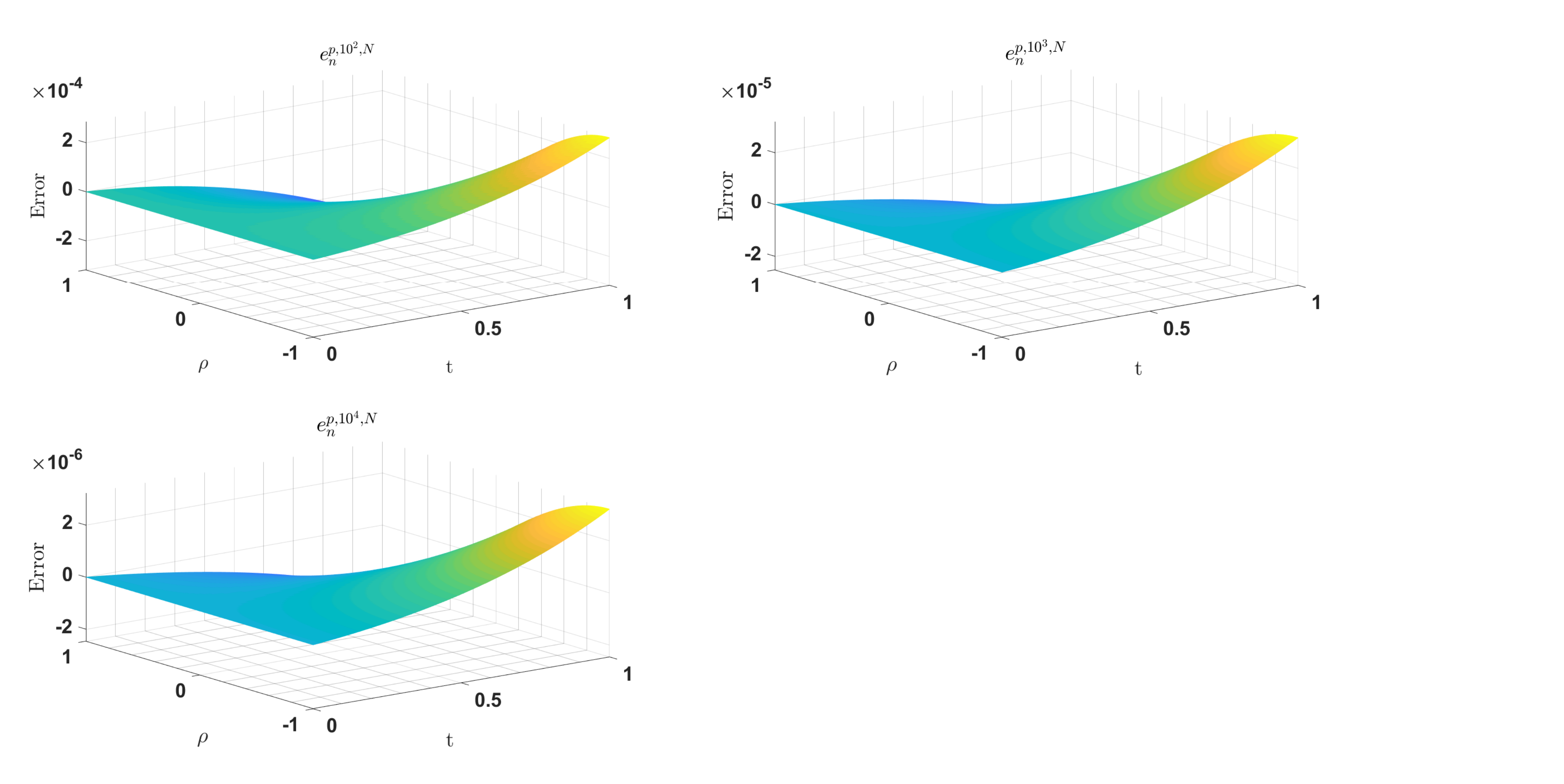}
\caption{\it{Error functions $e_n^{p,M,N}$ for $N=100$ and various values of M and $T=1$.}}
\label{TEN100}
\end{figure}
We have presented the maximum time-error and the computed order of convergence ($s$) using (\ref{ratio_remark})  by considering constant $N=100$ in Table \ref{time_error}. To better see the time-error of numerical results, Figure \ref{time_exact_semy} is presented. It can be observed from this figure that the errors decrease rapidly.  Also, in Figure \ref{time_exact_ratio} the computed order of convergence for numerical finite difference method is shown. It is shown that the finite difference method has almost $\mathcal{O}(h^2)$ error, as someone would expect from convergence Theorem \ref{convergencetheorem}. 
The fact is, in each time step, we have an error of order $\mathcal{O}(h^2)$ caused because of finite difference method and by large values of $M$ and $N$  the error of collocation method may tend to be far more than this. Due to this issue, it may be not suitable to indicate the error of collocation method in the presence of the one of finite difference. In turn, we consider the values of $M$ to be constant and assume the solution of the problem with $N = 600$ as an exact solution (because of the stability and convergence of the presented method) and present the maximum spectral-error with constants $M=1000$ and $M=100$ and various
values of $N$ in Table \ref{space_error_3}. To better see the spectral-error of numerical results, Figure \ref{space_error_fig} is presented. It is obviously observable from this figure that the errors decrease rapidly by increasing the collocation points and the spectral convergence rate is gained by the method.

Also, in order to show the stability of the method for the model (\ref{ex11})-(\ref{ex12}), $\epsilon_1(t,x)$ and $\epsilon_2(t,x)$ are added to the source term of the equations and the results obtain from the model (\ref{ex11})-(\ref{ex12}) and its perturbed are compared through the Table \ref{stabilityT} and Figure \ref{stabilityF} by considering different values of $\epsilon^*$ which is defined in Theorem \ref{stabilitytheorem}. As it is observable, the subtraction shows that a small perturbations in the right-hand side of the equation leads to arbitrarily small changes in the solution of the problem.

It is notable to say that in Figure \ref{stabilityF}, we have plotted the graph of the following error function for some values of $M$, $N$ and $\epsilon^*$
\begin{equation*}
e_n^{M,N,\epsilon^*}(p)=p_n^{ap}(p)-p_n^{ap,\epsilon^*}(p),
\end{equation*}
 where $p_n^{ap,\epsilon^*}$ is the solution of the perturbed model (\ref{ex11})-(\ref{ex12}) at the $n^{th}$ step of the finite difference method and in Table \ref{stabilityT}, the following error is presented
 \begin{equation*}
E_{\infty}^s=\max_{i=1,\cdots , N}\{\max_{j=1,\cdots M}\{\vert p_{N,M}(x_i,t_j)-p^s_{N,M}(x_i,t_j)\vert\}\},
 \end{equation*}
 where $p_{N,M}^s(x,t)$ is the solution of the perturbed model of (\ref{ex11})-(\ref{ex12}) using $N$ collocation points and $M$ time steps .
\begin{table}
\begin{center}
{\scriptsize{\begin{tabular}{lllllllllllll}
\hline
  &  & \multicolumn{1}{c}{$M=100$} & \multicolumn{1}{c}{} & \multicolumn{1}{c}{$M=200$} & \multicolumn{1}{c}{} & \multicolumn{1}{c}{$M=1000$} & \multicolumn{1}{c}{} & \multicolumn{1}{c}{$M=2000$} & \multicolumn{1}{c}{} & \multicolumn{1}{c}{$M=5000$} & \multicolumn{1}{c}{} & \multicolumn{1}{c}{$M=10000$} \\ 
\cline{1-1}\cline{3-3}\cline{5-5}\cline{7-7}\cline{9-9}\cline{11-11}\cline{13-13}
\multicolumn{1}{c}{Error of $p$} &  & \multicolumn{1}{c}{$1.0169e-03$} & \multicolumn{1}{c}{} & \multicolumn{1}{c}{$2.5075e-04$} & \multicolumn{1}{c}{} & \multicolumn{1}{c}{$9.9203e-06$} & \multicolumn{1}{c}{} & \multicolumn{1}{c}{$2.4766e-06$} & \multicolumn{1}{c}{} & \multicolumn{1}{c}{$3.9594e-07$} & \multicolumn{1}{c}{} & \multicolumn{1}{c}{$9.8958e-08$} \\ 
\multicolumn{1}{c}{\text{Rate of convergence ($s$)}} &  & \multicolumn{1}{c}{-} & \multicolumn{1}{c}{} & \multicolumn{1}{c}{$2.019912$} & \multicolumn{1}{c}{} & \multicolumn{1}{c}{$2.006831$} & \multicolumn{1}{c}{} & \multicolumn{1}{c}{$2.001978$} & \multicolumn{1}{c}{} & \multicolumn{1}{c}{$2.000896$} & \multicolumn{1}{c}{} & \multicolumn{1}{c}{$2.000407$} \\ 
\hline
\end{tabular}}}
\end{center}
\caption{\it{Maximum time-error with $N=100$, various $M$ and $T=1$ and also the rate of convergence with respect to the time variable .}}
\label{time_error}
\end{table}
\begin{figure}
\centering
\includegraphics[scale=.38]{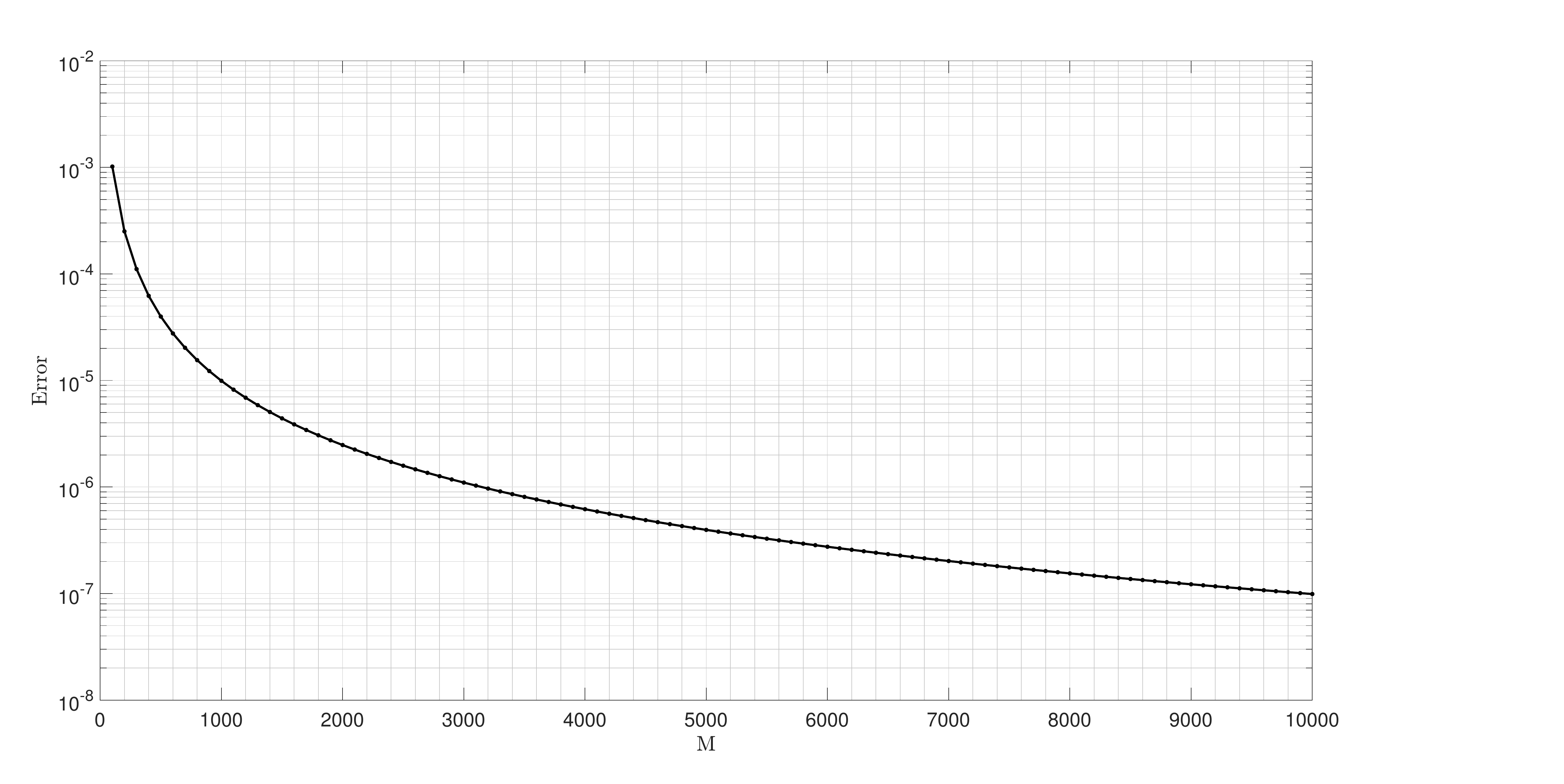}
\caption{\it{Maximum time-error with $N=100$, various values of M and $T=1$.}}
\label{time_exact_semy}
\end{figure}
\begin{figure}
\centering
\subfigure{\includegraphics[scale=.3]{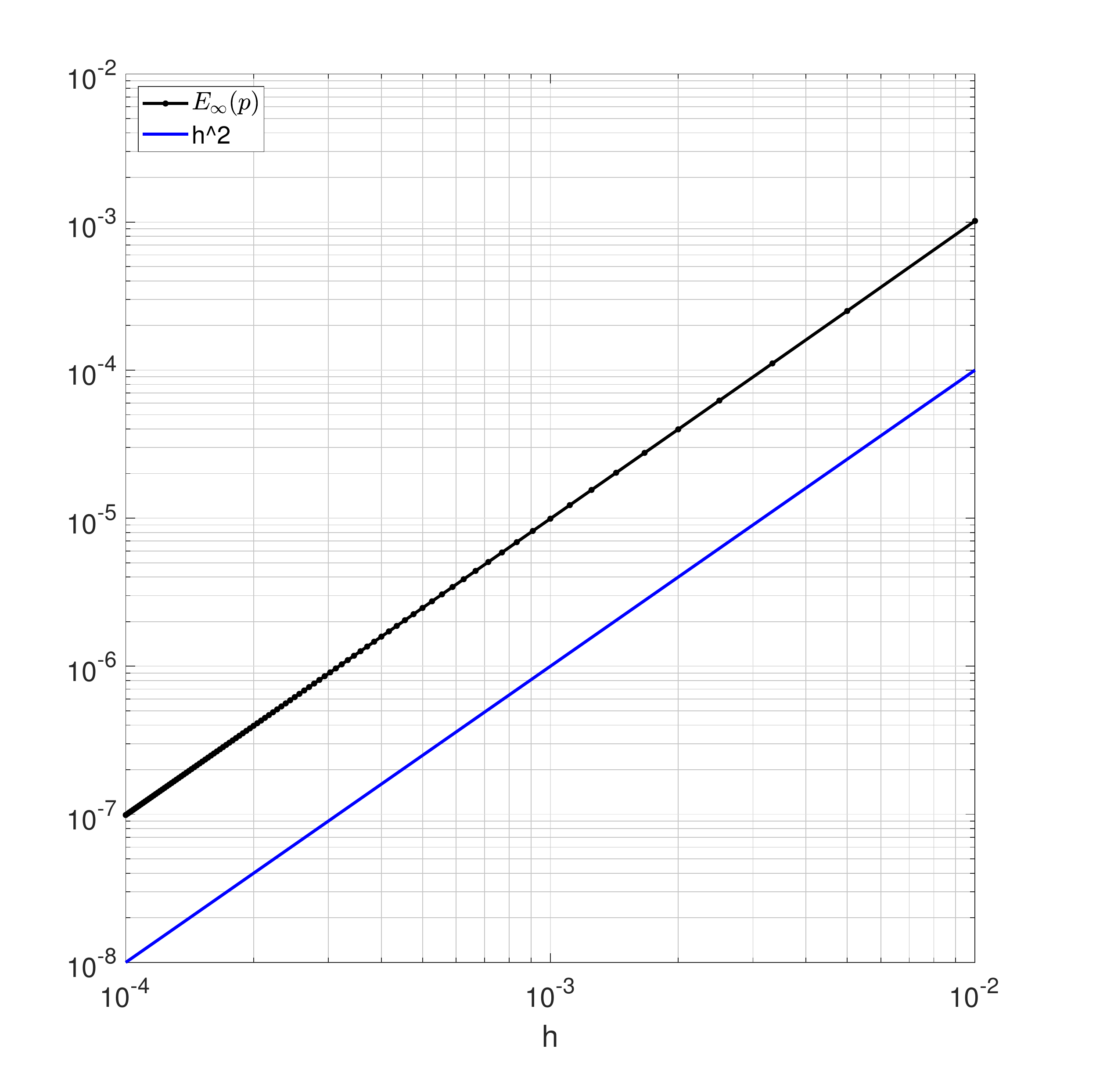}}
\subfigure{\includegraphics[scale=.36]{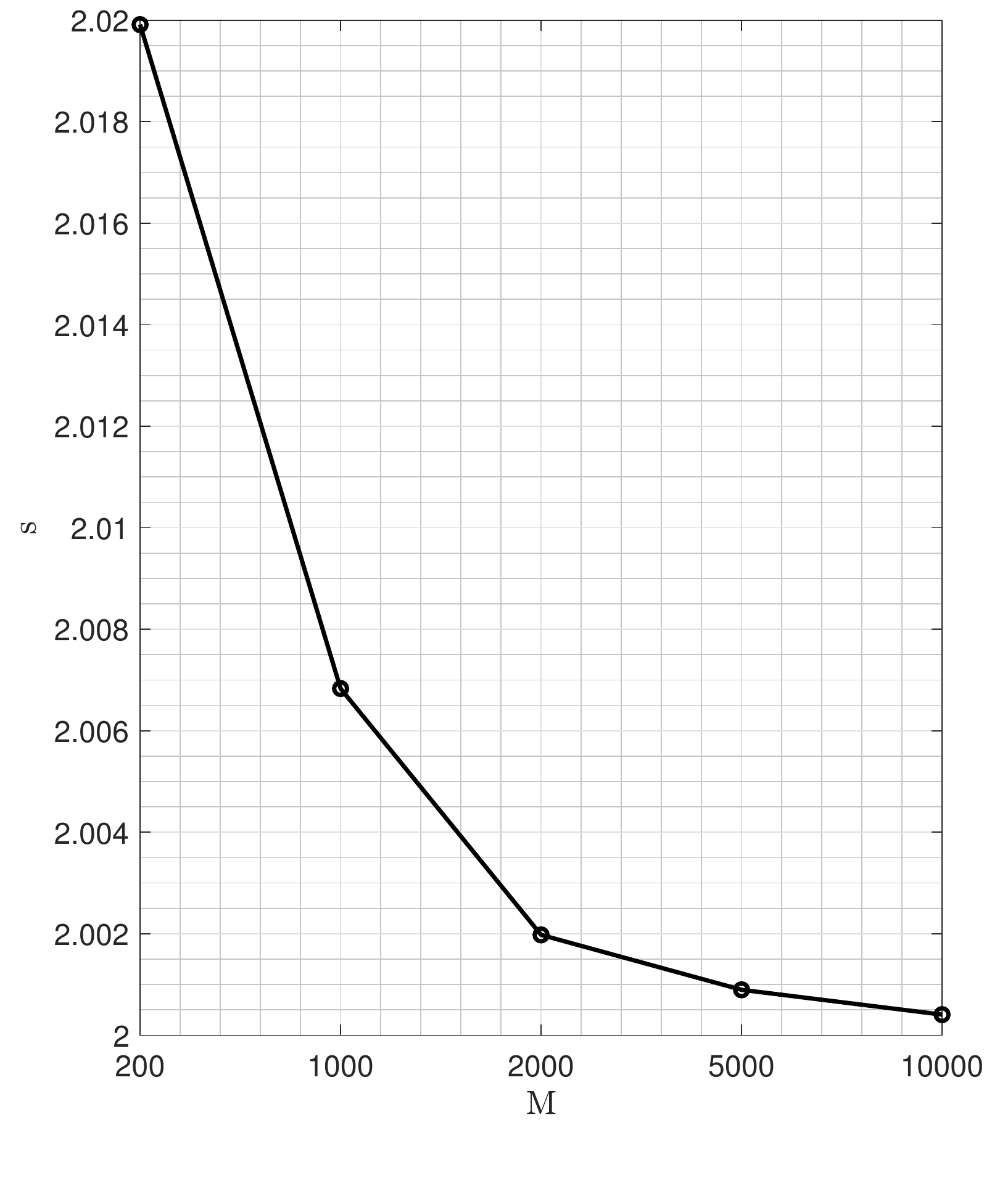}}
\caption{\it{The behaviour of time-error in LogLog scale (The left figure) and the rate of convergence (The right figure).}}
\label{time_exact_ratio}
\end{figure}
\begin{table}
\begin{center}
{\scriptsize{\begin{tabular}{lllllllllll}
\hline
\multicolumn{1}{c}{Error of $p$ with} & \multicolumn{1}{c}{} & \multicolumn{1}{c}{$N=10$} & \multicolumn{1}{c}{} & \multicolumn{1}{c}{$N=20$} & \multicolumn{1}{c}{} & \multicolumn{1}{c}{$N=100$} & \multicolumn{1}{c}{} & \multicolumn{1}{c}{$N=200$} & \multicolumn{1}{c}{} & \multicolumn{1}{c}{$N=300$} \\ 
\cline{1-1}\cline{3-3}\cline{5-5}\cline{7-7}\cline{9-9}\cline{11-11}
\multicolumn{1}{c}{$M=100$} & \multicolumn{1}{c}{} & \multicolumn{1}{c}{$1.368648e-08$} & \multicolumn{1}{c}{} & \multicolumn{1}{c}{$4.814340e-10$} & \multicolumn{1}{c}{} & \multicolumn{1}{c}{$8.211209e-13$} & \multicolumn{1}{c}{} & \multicolumn{1}{c}{$1.603162e-13$} & \multicolumn{1}{c}{} & \multicolumn{1}{c}{$1.807443e-13$} \\ 
\multicolumn{1}{c}{$M=1000$} & \multicolumn{1}{c}{} & \multicolumn{1}{c}{$1.442142e-09$} & \multicolumn{1}{c}{} & \multicolumn{1}{c}{$5.125921e-11$} & \multicolumn{1}{c}{} & \multicolumn{1}{c}{$5.369038e-13$} & \multicolumn{1}{c}{} & \multicolumn{1}{c}{$1.407762e-13$} & \multicolumn{1}{c}{} & \multicolumn{1}{c}{$3.383959e-13$} \\ 
\hline
\end{tabular}}}
\end{center}
\caption{\it{Maximum spectral-error with constants $M$, various values of $N$ and $T=1$.}}
\label{space_error_3}
\end{table}
\begin{figure}
\centering
\includegraphics[scale=.48]{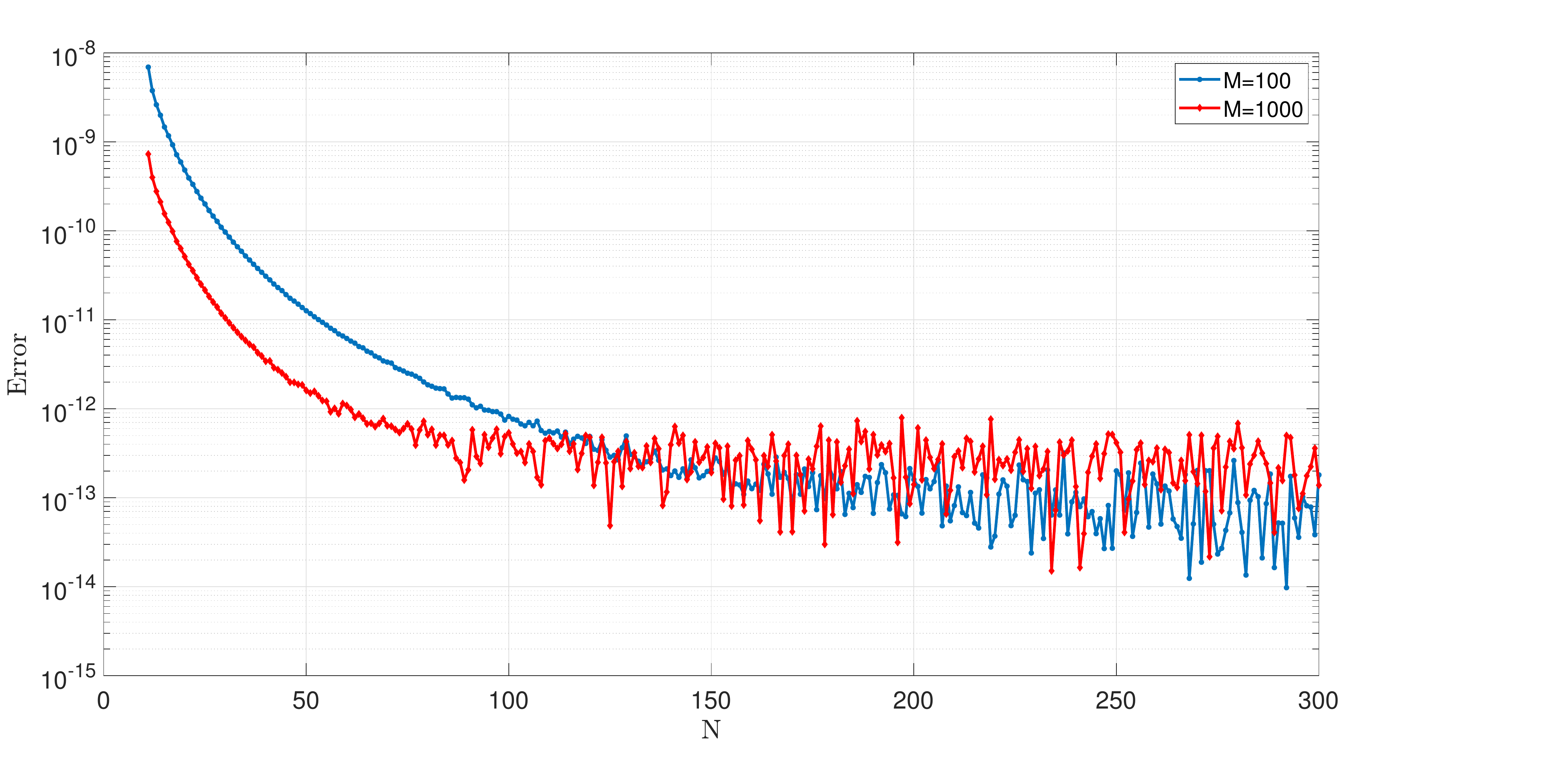}
\caption{\it{{Maximum spectral-error with constants $M$, various values of $N$ and $T=1$.}}}
\label{space_error_fig}
\end{figure}
\begin{table}
\begin{center}
{\scriptsize{\begin{tabular}{lllllllllll}
\hline
 &  & \multicolumn{1}{c}{$\epsilon^*=10^{-6}$} & \multicolumn{1}{c}{} & \multicolumn{1}{c}{$\epsilon^*=10^{-8}$} & \multicolumn{1}{c}{} & \multicolumn{1}{c}{$\epsilon^*=10^{-10}$} & \multicolumn{1}{c}{} & \multicolumn{1}{c}{$\epsilon^*=10^{-12}$} & \multicolumn{1}{c}{} & \multicolumn{1}{c}{$\epsilon^*=10^{-14}$} \\ 
\cline{1-1}\cline{3-3}\cline{5-5}\cline{7-7}\cline{9-9}\cline{11-11}
$M=1000 , N=10$ &  & \multicolumn{1}{c}{$6.1294e-07$} & \multicolumn{1}{c}{} & \multicolumn{1}{c}{$6.1294e-09$} & \multicolumn{1}{c}{} & \multicolumn{1}{c}{$6.1292e-11$} & \multicolumn{1}{c}{} & \multicolumn{1}{c}{ $5.8353e-13$} & \multicolumn{1}{c}{} & \multicolumn{1}{c}{$3.5527e-14$} \\ 
$M=2000 , N=20$ &  & \multicolumn{1}{c}{$6.1727e-07$} & \multicolumn{1}{c}{} & \multicolumn{1}{c}{$6.1727e-09$} & \multicolumn{1}{c}{} & \multicolumn{1}{c}{$6.1723e-11$} & \multicolumn{1}{c}{} & \multicolumn{1}{c}{$6.3505e-13$} & \multicolumn{1}{c}{} & \multicolumn{1}{c}{$4.5519e-14$} \\ 
$M=3000 , N=30$ &  & \multicolumn{1}{c}{$6.1749e-07$} & \multicolumn{1}{c}{} & \multicolumn{1}{c}{$6.1747e-09$} & \multicolumn{1}{c}{} & \multicolumn{1}{c}{$6.1625e-11$} & \multicolumn{1}{c}{} & \multicolumn{1}{c}{$5.2713e-13$} & \multicolumn{1}{c}{} & \multicolumn{1}{c}{$5.6177e-14$} \\ 
\hline
\end{tabular}}}
\end{center}
\caption{\it{The subtraction of the solution of (\ref{ex11})-(\ref{ex12}) and its perturbed with some values of $\epsilon^*$, $M$, $N$ and $T=1$.}}
\label{stabilityT}
\end{table}
\begin{figure}
\centering
\includegraphics[scale=.43]{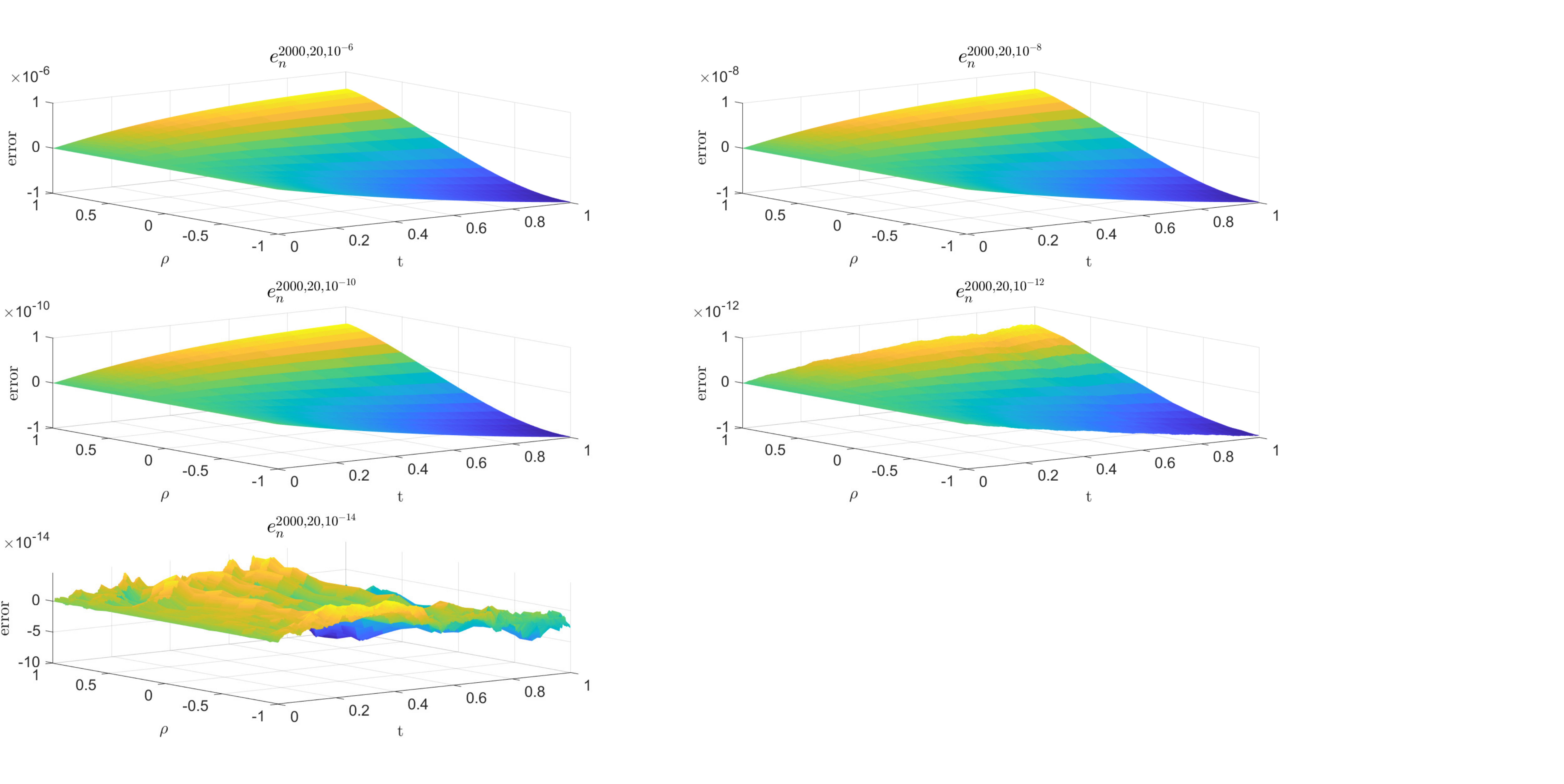}
\caption{\it{{The subtraction of the solution of (\ref{ex11})-(\ref{ex12}) and its perturbed with some values of $\epsilon^*$, $M=2000$, $N=20$ and $T=1$.}}}
\label{stabilityF}
\end{figure}
\newpage
\textbf{Example 2.} Consider the problem (\ref{ex11})-(\ref{ex12}) in Example 1 with the following exact solutions
\begin{equation*}
p(\rho,t)=\exp(t)(x^4-2x^2),~~R(t)=\dfrac{1}{(t+1)},~~ v(\rho,t)=-\dfrac{\sin(\pi x/2 )+1}{2(t+1)^2}.
\end{equation*} 
In Figure \ref{example2_n100M}, we have plotted the graph of the following error function for  various values of $M$.
\begin{equation*}
e_n^{p,M,N}=p_n^{ap}(\rho)-p_n(\rho).
\end{equation*}
\begin{figure}
\centering
\includegraphics[scale=.4]{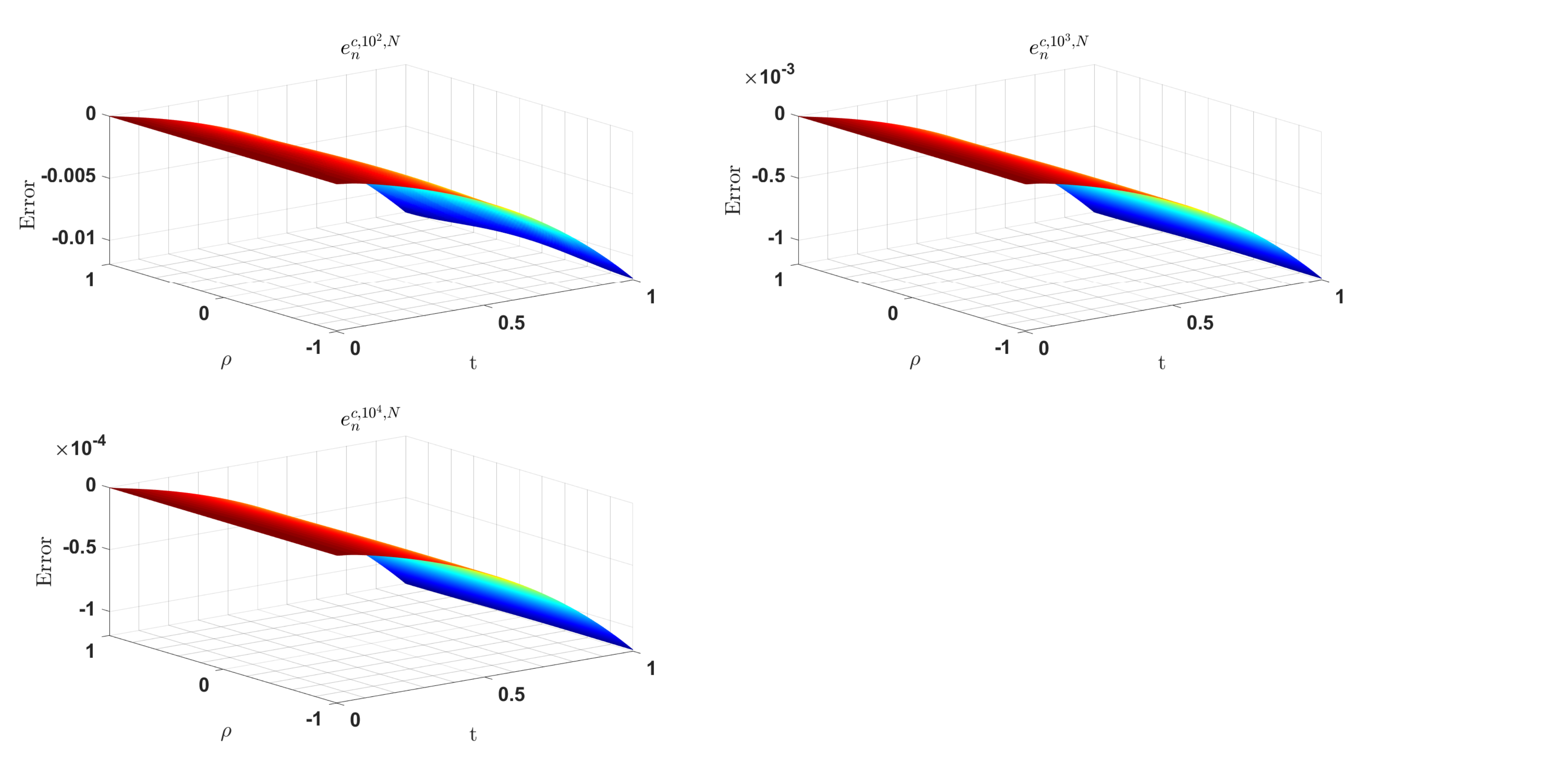}
\caption{\it{Error functions $e_n^{p,M,N}$ for $N=100$, various values of $M$ and $T=1$.}}
\label{example2_n100M}
\end{figure}
The maximum time-error and the computed order of convergence (s) using (\ref{ratio_remark})  by considering constant $N=100$ is presented in Table \ref{time_error_ex2} and Figures \ref{time_exact_semy_ex2} and \ref{time_exact_ratio_ex2}. It can be  observed from Table \ref{time_error_ex2} that the order of convergence of finite difference method is $\mathcal{O}(h^2)$.
We consider the values of $M$ to be constant and assume the solution of the problem with $N = 600$ as an exact solution and present the maximum spectral-error by constants $M=1000$ and $M=100$ and various
values of $N$ in Table \ref{space_error} and Figure \ref{spectral_ex2}. As a result, it can satisfy the expectation from the spectral convergenc rate, since it decreases rapidly by increasing the number of collocation points. 
\begin{table}
\begin{center}
{\scriptsize{\begin{tabular}{lllllllllllll}
\hline
 &  & \multicolumn{1}{c}{$M=100$} & \multicolumn{1}{c}{} & \multicolumn{1}{c}{$M=200$} & \multicolumn{1}{c}{} & \multicolumn{1}{c}{$M=1000$} & \multicolumn{1}{c}{} & \multicolumn{1}{c}{$M=2000$} & \multicolumn{1}{c}{} & \multicolumn{1}{c}{$M=5000$} & \multicolumn{1}{c}{} & \multicolumn{1}{c}{$M=10000$} \\ 
\cline{1-1}\cline{3-3}\cline{5-5}\cline{7-7}\cline{9-9}\cline{11-11}\cline{13-13}
\multicolumn{1}{c}{{Error of $p$}} &  & \multicolumn{1}{c}{$4.1583e-04$} & \multicolumn{1}{c}{} & \multicolumn{1}{c}{$1.0243e-04$} & \multicolumn{1}{c}{} & \multicolumn{1}{c}{$4.0496e-06$} & \multicolumn{1}{c}{} & \multicolumn{1}{c}{$1.1202e-06$} & \multicolumn{1}{c}{} & \multicolumn{1}{c}{$1.9094e-07$} & \multicolumn{1}{c}{} & \multicolumn{1}{c}{$4.0389e-08$} \\ 
\multicolumn{1}{c}{\text{Rate of convergence ($s$)}} &  & \multicolumn{1}{c}{-} & \multicolumn{1}{c}{} & \multicolumn{1}{c}{$ 2.0212768$} & \multicolumn{1}{c}{} & \multicolumn{1}{c}{$2.0073041$} & \multicolumn{1}{c}{} & \multicolumn{1}{c}{$2.0021163$} & \multicolumn{1}{c}{} & \multicolumn{1}{c}{$2.0009597$} & \multicolumn{1}{c}{} & \multicolumn{1}{c}{$2.0004299$} \\ 
\hline
\end{tabular}}}
\end{center}
\caption{\it{Maximum time-error with $N=100$, various $M$ and $T=1$ and the rate of convergence with respect to the time variable.}}
\label{time_error_ex2}
\end{table}
\begin{figure}
\centering
\includegraphics[scale=.38]{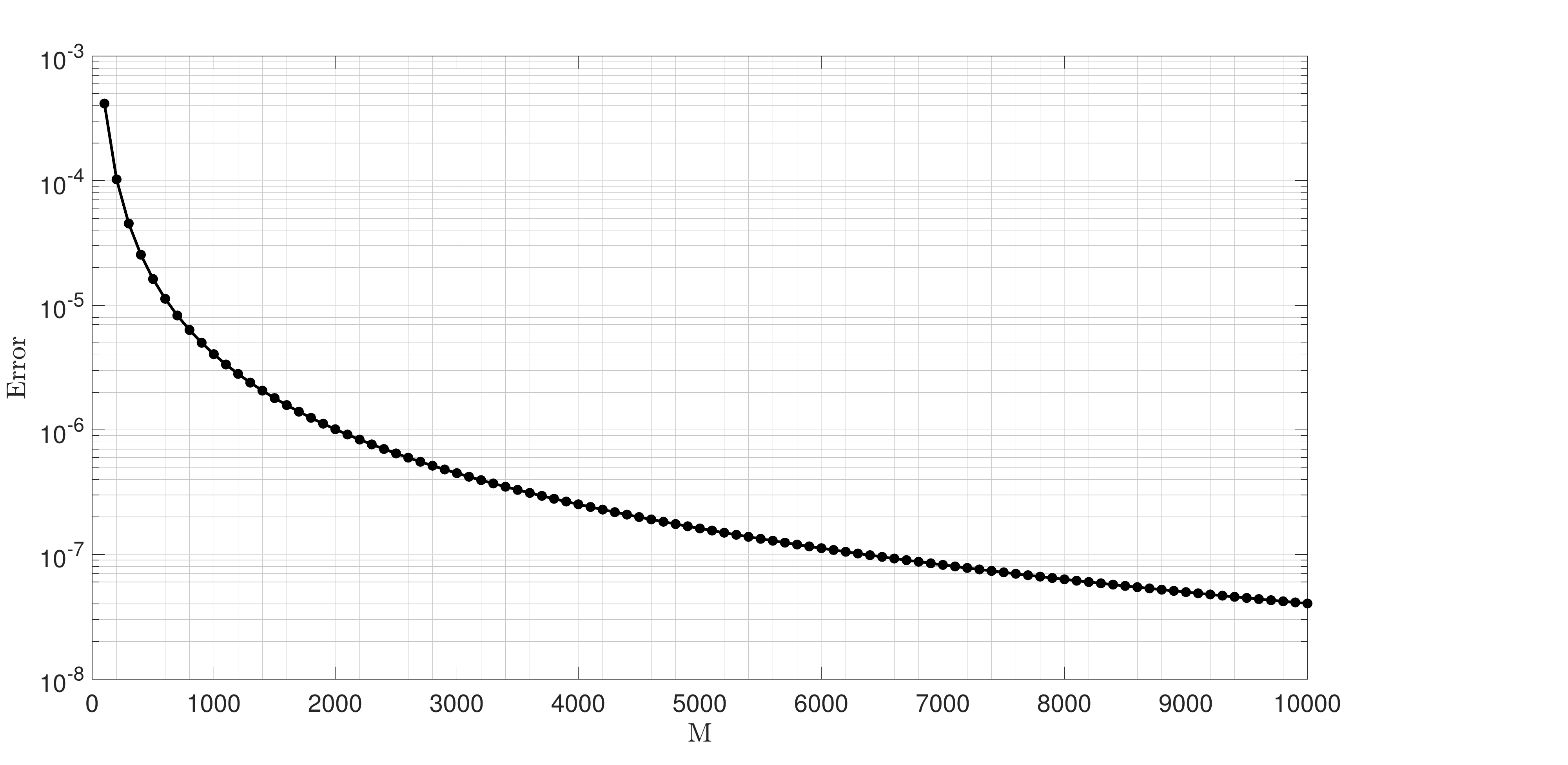}
\caption{\it{Maximum time-error with $N=100$ and various values of $M$.}}
\label{time_exact_semy_ex2}
\end{figure}
\begin{figure}
\centering
\subfigure{\includegraphics[scale=.3]{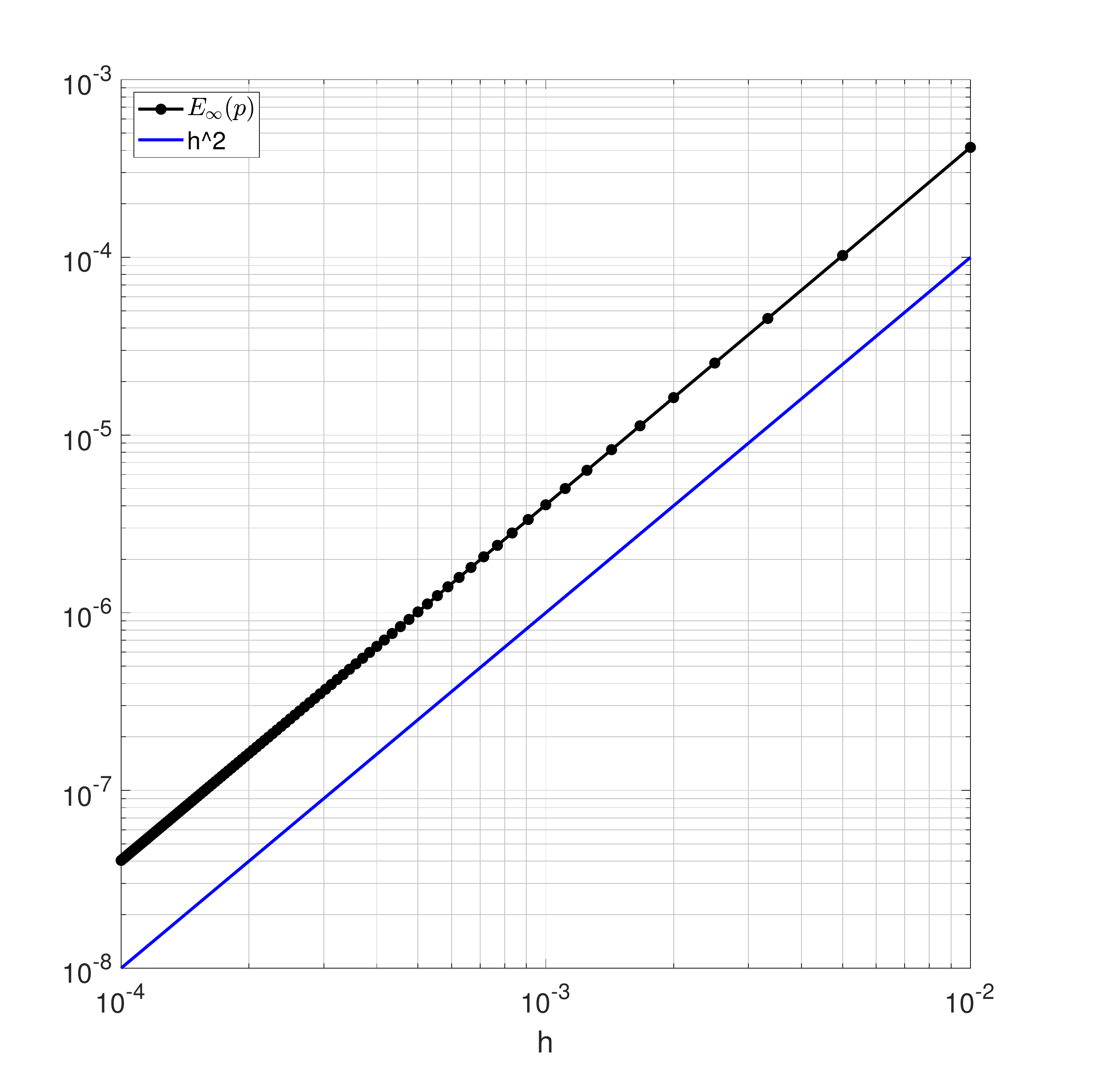}}
\subfigure{\includegraphics[scale=.3]{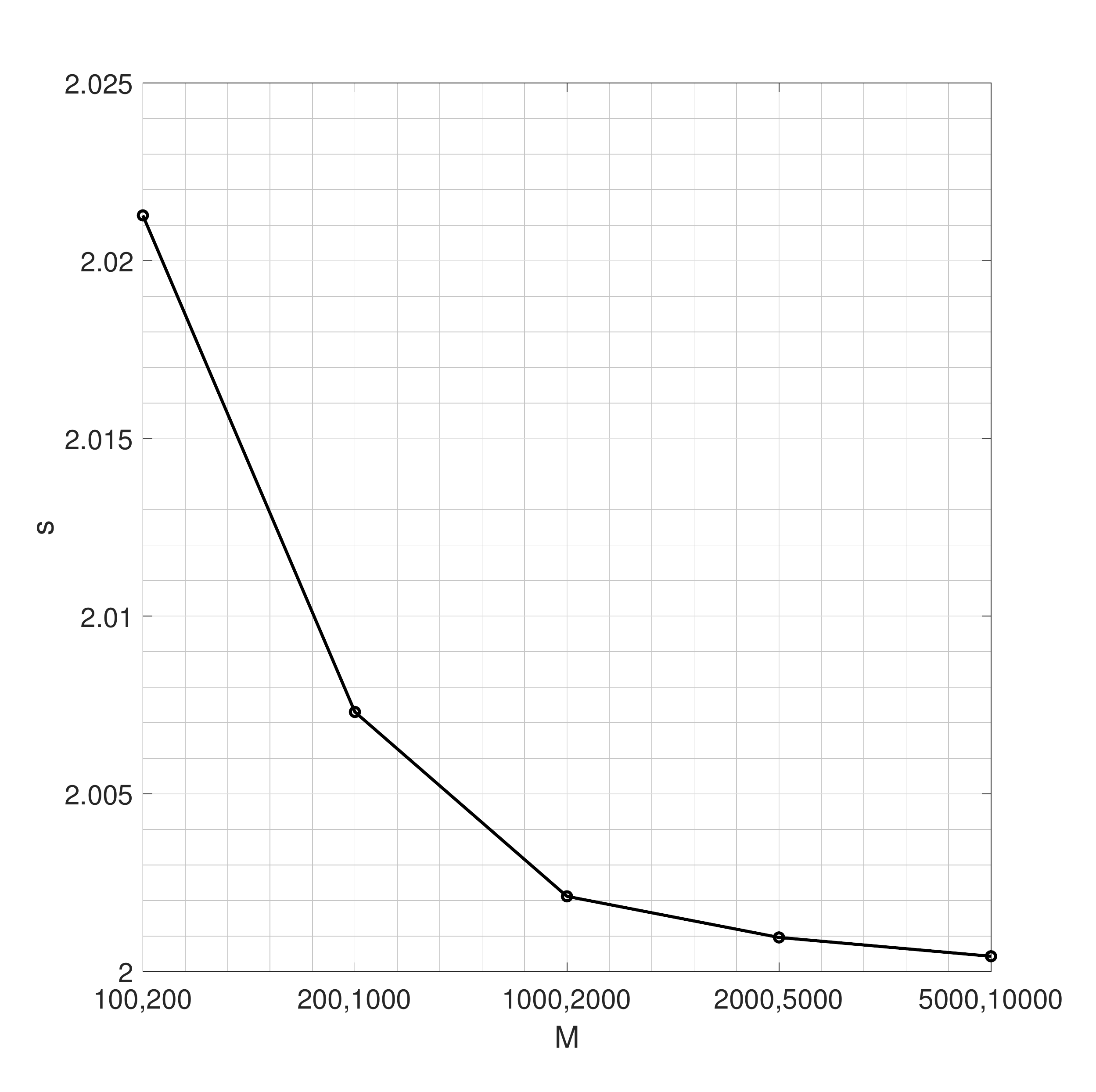}}
\caption{\it{The behaviour of time-error in LogLog scale (The left figure) and the rate of convergence (The right figure).}}
\label{time_exact_ratio_ex2}
\end{figure}
\begin{table}
\begin{center}
{\scriptsize{\begin{tabular}{lllllllllll}
\hline
\multicolumn{1}{c}{Error of p with} & \multicolumn{1}{c}{} & \multicolumn{1}{c}{$N=10$} & \multicolumn{1}{c}{} & \multicolumn{1}{c}{$N=20$} & \multicolumn{1}{c}{} & \multicolumn{1}{c}{$N=100$} & \multicolumn{1}{c}{} & \multicolumn{1}{c}{$N=200$} & \multicolumn{1}{c}{} & \multicolumn{1}{c}{$N=300$} \\ 
\cline{1-1}\cline{3-3}\cline{5-5}\cline{7-7}\cline{9-9}\cline{11-11}
\multicolumn{1}{c}{$M=100$} & \multicolumn{1}{c}{} & \multicolumn{1}{c}{$1.368648e-08$} & \multicolumn{1}{c}{} & \multicolumn{1}{c}{$4.814340e-10$} & \multicolumn{1}{c}{} & \multicolumn{1}{c}{$8.211209e-13$} & \multicolumn{1}{c}{} & \multicolumn{1}{c}{$1.603162e-13$} & \multicolumn{1}{c}{} & \multicolumn{1}{c}{$1.007443e-13$} \\ 
\multicolumn{1}{c}{$M=1000$} & \multicolumn{1}{c}{} & \multicolumn{1}{c}{$1.442142e-09$} & \multicolumn{1}{c}{} & \multicolumn{1}{c}{$5.125921e-11$} & \multicolumn{1}{c}{} & \multicolumn{1}{c}{$5.369038e-13$} & \multicolumn{1}{c}{} & \multicolumn{1}{c}{$1.407762e-13$} & \multicolumn{1}{c}{} & \multicolumn{1}{c}{$1.083959e-13$} \\ 
\hline
\end{tabular}}}
\end{center}
\caption{\it{Maximum spectral-error with constants $M$, various values of $N$ and $T=1$.}}
\label{space_error}
\end{table}
\begin{figure}
\centering
\includegraphics[scale=.48]{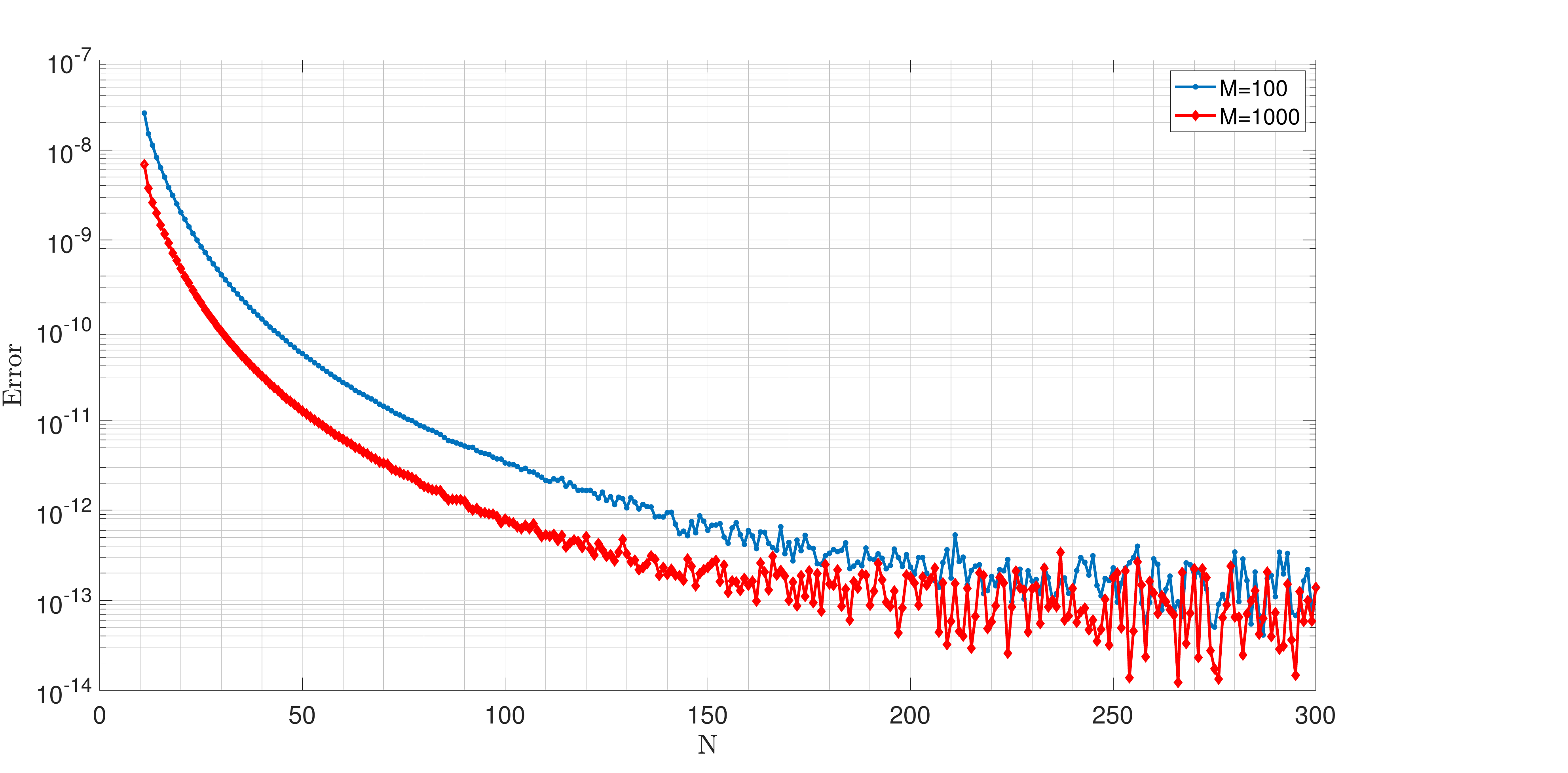}
\caption{\it{{Maximum spectral-error with constants $M$, various values of $N$ and $T=1$.}}}
\label{spectral_ex2}
\end{figure}
\newpage
\textbf{Example 3.} Now based on the efficiency of methods illustrated in Examples 1 and 2, in the following we intend to solve the problem (\ref{transformedp})-(\ref{transformedR}) and present the numerical errors by considering a solution obtained from a fine mesh. First of all, Regarding the fact that we do not have any exact solution of the problem (\ref{transformedp})-(\ref{transformedR}), to ensure about the reliability of the solution we have obtained, the "PDEPE" Matlab package is considered which is suitable for solving parabolic equations along with elliptic ones. In Figure \ref{exact_ex3} the error of the solution of the problem (\ref{transformedp})-(\ref{transformedR}) obtained from the mentioned finite difference-collocation method by considering the solution of the problem given from the "PDEPE" package using a fine mesh as an exact solution is illustrated.
To do so, we have presented the maximum time-error and the computed order of convergence (s) using (\ref{ratio_remark})  by considering constant $N=100$  and assuming the solution of the problem with $M = 20000$ as an exact solution in Table \ref{time_error3}. To better see the time-error of numerical results, Figure \ref{time_error_ex3} is presented. It can be observed that the error decreases by increasing the  number of time steps. Also, in Figure \ref{time_ratio_ex3} the computed order of convergence for numerical finite difference method is shown. It is shown that the finite difference method has almost $\mathcal{O}(h^2)$ error, as someone would expect from convergence Theorem \ref{convergencetheorem}. 
In order to illustrate the error of spectral method, the values of $M$ is considered to be constant and we assume the solution of the problem with $N = 600$ as an exact solution and present the maximum spectral-error by constant  $M=100$ and various
values of $N$ in Table \ref{space_error_ex3_t}. To better see the spectral-error of numerical results, Figure \ref{space_error_ex3} is presented and it shows the convergence rate of the collocation method as it decreases rapidly by increasing the collocation points.
\begin{figure}
\centering
\includegraphics[scale=.45]{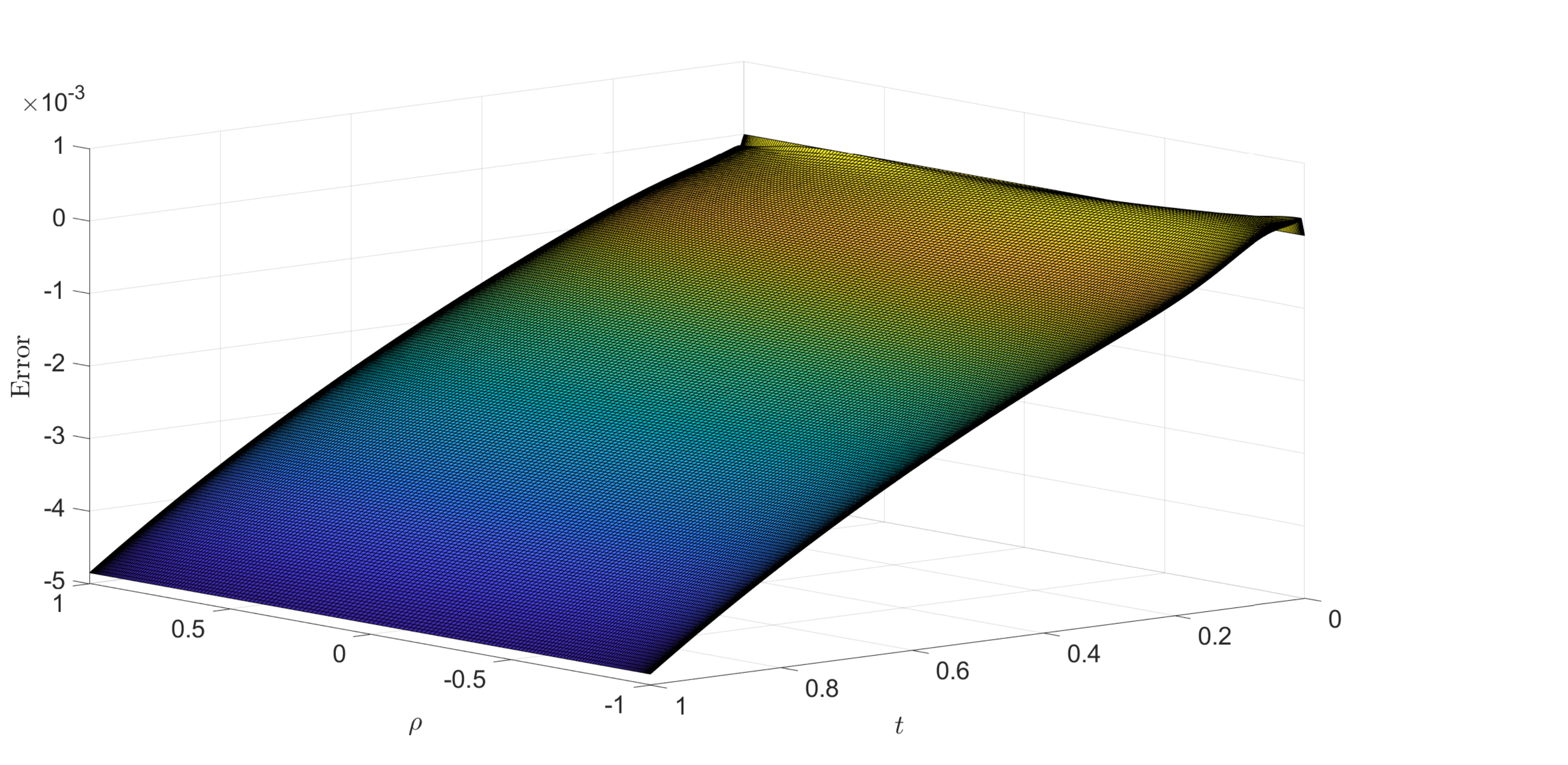}
\caption{\it{{Maximum error of the solution obtained from the finite difference/collocation method by $M=200$, $N=200$ and $T=1$ by considering the solution obtained from the "PDEPE" package by $M=2000$ and $N=2000$ as an exact solution.}}}
\label{exact_ex3}
\end{figure}
\begin{table}
\begin{center}
{\scriptsize{\begin{tabular}{lllllllllllll}
\hline
{} &  & \multicolumn{1}{c}{$M=100$} & \multicolumn{1}{c}{} & \multicolumn{1}{c}{$M=200$} & \multicolumn{1}{c}{} & \multicolumn{1}{c}{$M=1000$} & \multicolumn{1}{c}{} & \multicolumn{1}{c}{$M=2000$} & \multicolumn{1}{c}{} & \multicolumn{1}{c}{$M=5000$} & \multicolumn{1}{c}{} & \multicolumn{1}{c}{$M=10000$} \\ 
\cline{1-1}\cline{3-3}\cline{5-5}\cline{7-7}\cline{9-9}\cline{11-11}\cline{13-13}
\multicolumn{1}{c}{Error of $p$} &  & \multicolumn{1}{c}{$1.3593e-04$} & \multicolumn{1}{c}{} & \multicolumn{1}{c}{$6.7389e-05$} & \multicolumn{1}{c}{} & \multicolumn{1}{c}{$1.2896e-05$} & \multicolumn{1}{c}{} & \multicolumn{1}{c}{$6.1068e-06$} & \multicolumn{1}{c}{} & \multicolumn{1}{c}{$2.0351e-06$} & \multicolumn{1}{c}{} & \multicolumn{1}{c}{$6.7834e-07$} \\ 
\multicolumn{1}{c}{\text{Rate of convergence ($s$)}} &  & \multicolumn{1}{c}{-} & \multicolumn{1}{c}{} & \multicolumn{1}{c}{$1.0123358$} & \multicolumn{1}{c}{} & \multicolumn{1}{c}{$1.0505686$} & \multicolumn{1}{c}{} & \multicolumn{1}{c}{$1.1083728$} & \multicolumn{1}{c}{} & \multicolumn{1}{c}{$1.329036$} & \multicolumn{1}{c}{} & \multicolumn{1}{c}{$1.9705147$} \\ 
\hline
\end{tabular}}}
\end{center}
\caption{\it{Maximum time-error with $N=100$, various values of  of $M$ and $T=1$.}}
\label{time_error3}
\end{table}
\begin{figure}
\centering
\includegraphics[scale=.38]{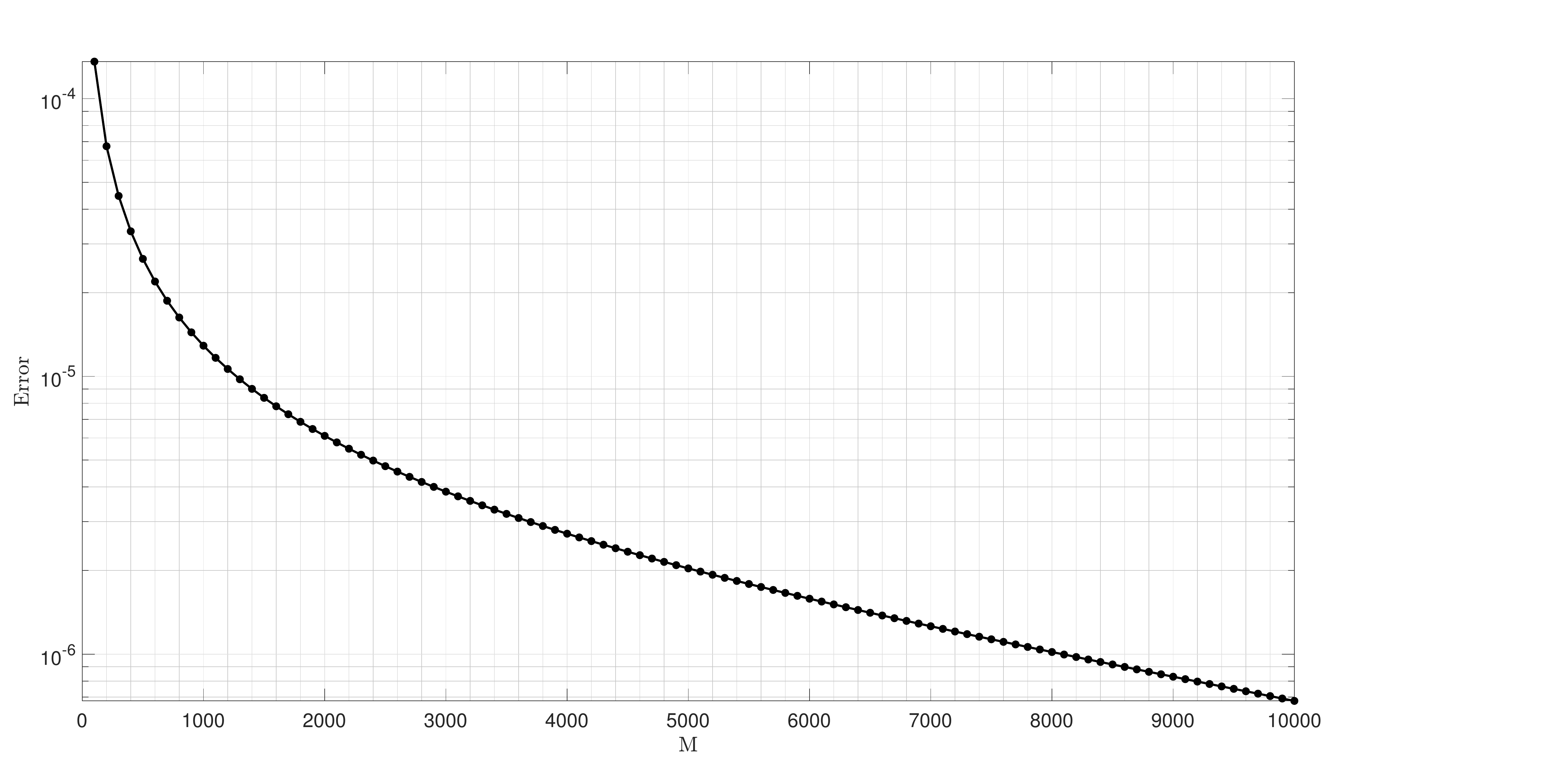}
\caption{\it{{Maximum time-error with constant $N$,  various values of  $M$ and $T=1$.}}}
\label{time_error_ex3}
\end{figure}
\begin{figure}
\centering
\subfigure{\includegraphics[scale=.3]{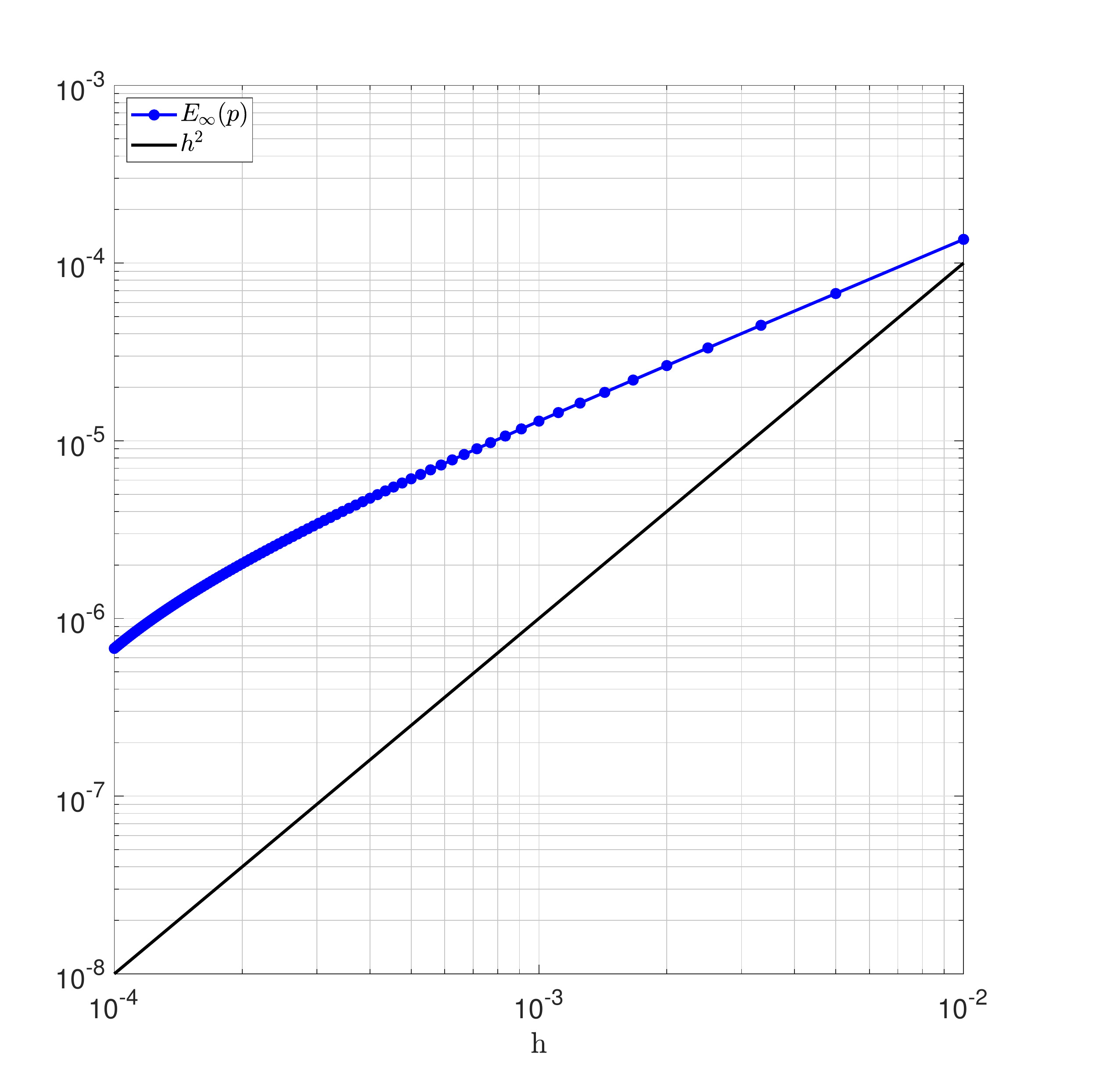}}
\subfigure{\includegraphics[scale=.3]{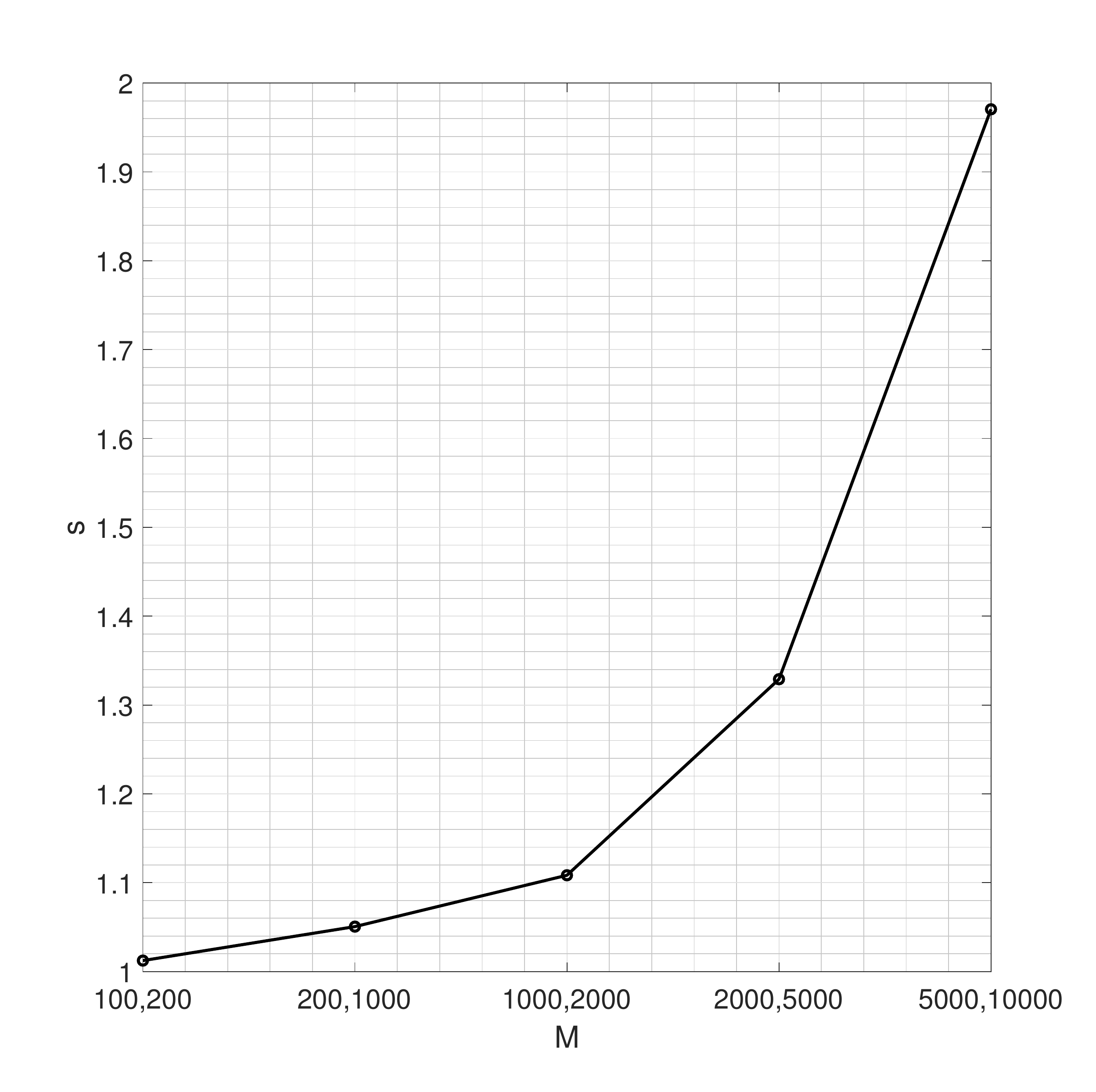}}
\caption{\it{The behaviour of time-error in LogLog scale (The left figure) and the rate of convergence (The right figure).}}
\label{time_ratio_ex3}
\end{figure}
\begin{table}
\begin{center}
{\scriptsize{\begin{tabular}{lllllllllll}
\hline
\multicolumn{1}{c}{Error of p with} & \multicolumn{1}{c}{} & \multicolumn{1}{c}{N=10} & \multicolumn{1}{c}{} & \multicolumn{1}{c}{$N=20$} & \multicolumn{1}{c}{} & \multicolumn{1}{c}{$N=100$} & \multicolumn{1}{c}{} & \multicolumn{1}{c}{$N=200$} & \multicolumn{1}{c}{} & \multicolumn{1}{c}{$N=300$} \\ 
\hline
\multicolumn{1}{c}{$M=100$} & \multicolumn{1}{c}{} & \multicolumn{1}{c}{$7.8479688e-09$} & \multicolumn{1}{c}{} & \multicolumn{1}{c}{$3.727724e10$} & \multicolumn{1}{c}{} & \multicolumn{1}{c}{$6.243964e-13$} & \multicolumn{1}{c}{} & \multicolumn{1}{c}{$3.908324e-14$} & \multicolumn{1}{c}{} & \multicolumn{1}{c}{$7.338763e-15$} \\ 
\hline
\end{tabular}}}
\end{center}
\caption{\it{Maximum spectral-error with constant $M$, various values of $N$ and $T=1$.}}
\label{space_error_ex3_t}
\end{table}
\begin{figure}
\centering
\includegraphics[scale=.38]{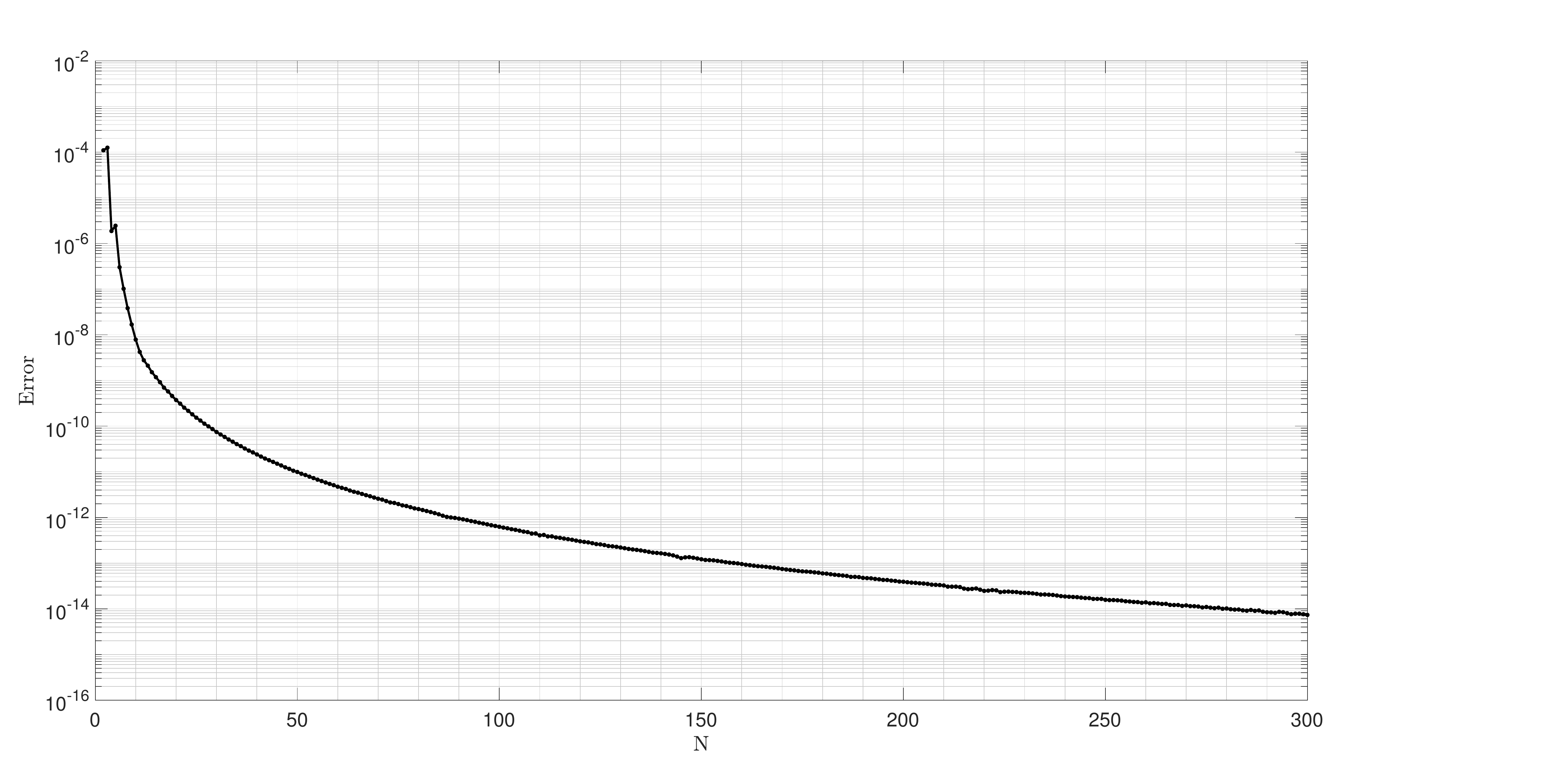}
\caption{\it{ Maximum spectral-error with constant $M$, various values of $N$ and $T=1$.}}
\label{space_error_ex3}
\end{figure}
\newpage
\section{Conclusion}
In this paper, we have considered a nonlinear coupled free boundary problem modelling the growth of prostate tumor including
two reaction-diffusion equations describing the diffusion of androgen-dependent and androgen-independent cells in the tumor. 
 However, ing, especially when it comes to biological problems often produces nonlinear differential equations. Therefore, we are not always be able to obtain the exact solution of these equations, developing numerical techniques to solve these equations is a pressing need. In this study, a mathematical model of prostate tumor is solved numerically and the convergence and stability analysis are presented. For the reader’s convenience, we give the main contributions of this study as follows \\
$\bullet$ As the mentioned mathematical model is a free boundary model and the classical methods are not be efficient in solving these kind of problems, and since it is observable that front fixing method is highly efficient in applying to problems with regular geometries along with the mesh-based methods, so, in this article, we use the front fixing method to convert the free boundary problem (\ref{transformedp})-(\ref{transformedR}) to a fix one.\\
$\bullet$ We have used Taylor theorem, in order to both linearize the equations and construct new second-order non-classical discretization formula to approximate time discretization (Finite difference method).\\
$\bullet$ In this article, we use spectral collocation method in space. To construct trial functions which satisfy the boundary conditions, a linear combination of classical orthogonal polynomials (Legendre polynomials) to construct trial functions is used.\\
$\bullet$ Moreover, in terms of analytical aspects, the convergence and stability of the presented method is proved (See Theorem \ref{convergencetheorem} and Theorem \ref{stabilitytheorem}) and the order of convergence is presented.\\
$\bullet$ In order to indicate the efficiency of methods presented to solve the model numerical-wise, the numerical results are presented in the format of tables and figures. it is shown that the finite difference method displays an $\mathcal{O}(h^2)$ order of convergence, as one would expect from convergence Theorem \ref{convergencetheorem} (See Figure \ref{time_exact_ratio}, \ref{time_exact_ratio_ex2} and \ref{time_ratio_ex3}), and the spectral-error shows that using the collocation method, the results are converging to the exact solution.\\

\end{document}